\documentclass[reqno, 11pt]{article}
\usepackage[all]{xy}
\usepackage{latexsym}
\usepackage{amsmath}
\usepackage[latin1]{inputenc}
\usepackage[T1]{fontenc}
\usepackage[francais]{babel}
\usepackage{amssymb,amscd,amsthm}
\usepackage[mathscr]{eucal}
\usepackage{bbm}

\providecommand{\cD}{\mathcal{D}}
\providecommand{\cH}{\mathcal{H}}

\providecommand{\cK}{\mathcal{K}}

\providecommand{\cO}{\mathcal{O}}
\providecommand{\cP}{\mathcal{P}}

\providecommand{\F}{\mathbb F}

\providecommand{\Q}{\mathbb Q}
\providecommand{\Z}{\mathbb Z}

\providecommand{\rInj}{\mathrm{Inj}}
\providecommand{\rsoc}{\mathrm{soc}}
\providecommand{\JH}{\mathrm{JH}}

\providecommand{\p}{\mathfrak {p}}

\providecommand{\Fr}{\mathrm{Fr}}

\def\lra{\longrightarrow}

\providecommand{\Hom}{\mathrm{Hom}}
\providecommand{\Ext}{\mathrm{Ext}}
\providecommand{\GL}{\mathrm {GL}}
\providecommand{\SL}{\mathrm {SL}}
\providecommand{\Gal}{\mathrm{Gal}}
\providecommand{\id}{\mathrm{Id}}

\providecommand{\rad}{\mathrm{rad}}
\providecommand{\Sym}{\mathrm{Sym}}
\providecommand{\Rep}{\underline{\mathrm {Rep}}}
\providecommand{\Ind}{\mathrm {Ind}}

\providecommand{\xto}[1][]{\xrightarrow{#1}}

\providecommand{\simto}{
\xto[\sim]} 

\providecommand{\summ}{\sum\limits}

\providecommand{\bFp}{\overline{\F}_p}
\providecommand{\bQp}{\overline{\Q}_p}

\providecommand{\matr}[4]{\begin{pmatrix}{#1}&{#2}\\
{#3}&{#4}\end{pmatrix}}
\providecommand{\smatr}[4]{\bigl(\begin{smallmatrix} {#1}& {#2}\\
{#3}&{#4}\end{smallmatrix}\bigl)}

\providecommand{\DIAG}{\mathcal{DIAG}}

\providecommand{\red}{\mathrm{r\acute{e}d}}
\providecommand{\irr}{\mathrm{irr}}
\providecommand{\ra}{\rightarrow} 

\providecommand{\vv}{\vspace{2mm}}
\providecommand{\cP}{\mathcal{P}}

\providecommand{\cD}{\mathcal{D}}

\providecommand{\cS}{\mathcal{S}}
\providecommand{\cPx}{\mathcal{P}(x_0,\cdots,x_{f-1})}

\providecommand{\cRDx}{\mathcal{RD}(x_0,\cdots,x_{f-1})}
\providecommand{\cIDx}{\mathcal{ID}(x_0,\cdots,x_{f-1})}
\providecommand{\cPy}{\mathcal{P}(y_0,\cdots,y_{f-1})}
\providecommand{\cIy}{\mathcal{I}(y_0,\cdots,y_{f-1})}

\providecommand{\ligne}{\textbf{---}}
\providecommand{\rcosoc}{\mathrm{cosoc}}
\providecommand{\lon}{s} 

\newtheorem{theorem}{Théorème}[section]
\newtheorem{lemma}[theorem]{Lemme}
\newtheorem{cor}[theorem]{Corollaire}
\newtheorem{prop}[theorem]{Proposition}
\newtheorem{defn}[theorem]{Définition}

\newtheorem{rem}[theorem]{Remarque}
\newtheorem{exem}[theorem]{Exemple}

\begin{document}

\title{Sur quelques représentations supersingulières de $\GL_2(\Q_{p^f})$}

\author{Yongquan Hu}
\date{}
\maketitle

\vspace{-10mm}

\hspace{5cm}\hrulefill\hspace{5.5cm} \vspace{5mm}

\textbf{Résumé}--Soit $p\geq 3$ un nombre premier, $f\geq 1$ un entier
et $\Q_{p^f}$ l'extension finie non ramifiée de $\Q_{p}$ de degré $f$.
D'après \cite{BP}, à une représentation continue semi-simple générique
$\Gal(\bQp/\Q_{p^f})\ra \GL_2({\bFp})$, on sait associer une
famille de représentations lisses admissibles de $\GL_2(\Q_{p^f})$
à coefficients dans $\bFp$ dont des paremètres
sont données. Dans cet article, on montre qu'il y a beaucoup de
paramètres que l'on sait.
\\

\tableofcontents

\section{Introduction}
 Fixons $p$ un nombre premier. Serre a conjectur\'e (\cite{Se2}), il y a d\'ej\`a 20 ans, que toute repr\'esentation continue, irr\'eductible, impaire $\Gal(\overline{\Q}/\Q)\ra \GL_2(\bFp)$ est \emph{modulaire} au sens qu'elle provient d'une forme modulaire parabolique primitive.
Dans \cite{BDJ}, Buzzard, Diamond et Jarvis généralisent cette conjecture en rempla\c{c}ant $\Q$ par un corps de nombres totalement réel qui est non ramifié en $p$.
Pour notre propos, la conjecture de Buzzard-Diamond-Jarvis (abr\'eg\'e par BDJ) peut se formuler grossièrement comme suit (cf. \cite[conjecture 4.7]{BDJ}). Soient $F$ un corps de nombres totalement r\'eel non ramifi\'e en $p$ et $\rho:\Gal(\overline{\Q}/F)\ra\GL_2(\bFp)$ une repr\'esentation continue irr\'eductible et totalement impaire. Pour chaque place $\nu$ de $F$ divisant $p$, d'apr\`es le travail de \cite{BDJ}, on peut associer \`a $\rho|_{\Gal(\bQp/F_{\nu})}$ (o\`u $F_{\nu}$ est le compl\'et\'e de $F$ en $\nu$) un ensemble de repr\'esentations irr\'eductibles de $\GL_2(\cO_{F_{\nu}})$ (o\`u $\cO_{F_{\nu}}$ est l'anneau des entiers de $F_{\nu}$). Alors (une version de) la conjecture de BDJ dit que ce sont exactement les repr\'esentations irr\'eductibles qui apparaissent en sous-objet dans une certaine $\bFp$-repr\'esentation lisse admissible $\pi_{\nu}(\rho)$ de $\GL_2(F_{\nu})$ associ\'ee \`a $\rho$.


Des résultats partiels sur la conjecture de BDJ ont été  obtenus par Gee (\cite{Ge}).
De plus, sous une hypoth\`ese suppl\'ementaire sur $\rho|_{\Gal(\bQp/F_{\nu})}$ appel\'ee <<g\'en\'ericit\'e>> (\cite[définition 11.7]{BP}), il est espéré  (cf. \cite[remarque 4.8]{BDJ}) que toutes les multiplicit\'es des repr\'esentations de $\GL_2(\cO_{F_{\nu}})$ apparaissant en sous-objet dans $\pi_{\nu}(\rho)$ sont \'egales \`a $1$. On peut donc se demander si l'on peut construire abstraitement des représentations lisses admissibles de $\GL_2(F_{\nu})$ dont le $\GL_2(\cO_{F_{\nu}})$-socle, i.e. la plus grande sous-$\GL_2(\cO_{F_{\nu}})$-repr\'esentation semi-simple, est exactement la somme directe des repr\'esentations irr\'eductibles dict\'ees par la conjecture.
C'est ce qui est fait dans \cite{BP}. \vv

Notre probl\`eme \'etant local, on d\'esigne d\'esormais par $F$ l'extension $\Q_{p^f}$ de $\Q_p$ (l'unique extension non ramifi\'ee de degr\'e $f\geq 1$) et $\cO_F$ son anneau des entiers.  On note
\[G=\GL_2(F),\ \ \ K=\GL_2(\cO_F),\]
$I$ le sous-groupe d'Iwahori de $K$,  $N$ le normalisateur de $I$
dans $G$, et $I_1\subset I$ (resp. $K_1\subset K$) le sous-groupe des matrices unipotentes supérieures (resp. égales à l'identité) modulo $p$.

Soit $\rho:\Gal(\bQp/F)\ra\GL_2(\bFp)$ une repr\'esentation continue semi-simple g\'en\'erique et telle que $p\in Z$ agisse trivialement sur $\det(\rho)$. Notons $\cD(\rho)$ l'ensemble des repr\'esentations irr\'eductibles de $K$ sur $\bFp$, ou de mani\`ere \'equivalente de $K/K_1$ sur $\bFp$, associ\'e \`a $\rho$ dans \cite{BDJ}. On appelle \emph{poids de Diamond} les \'el\'ements de $\cD(\rho)$. \`A partir de $\cD(\rho)$, on peut construire une famille de diagrammes (au sens de \cite[\S9]{BP}) $D(\rho,r)=(D_0(\rho),D_1(\rho),r)$ comme suit (en faisant agir $p\in Z$ trivialement):
\begin{enumerate}
\item[(i)]
$D_0(\rho)$ est la plus grande représentation de $K/K_1$ telle que $\rsoc_K(D_0(\rho))$, le $K$-socle de $D_0(\rho)$, est isomorphe \`a $\oplus_{\sigma\in\cD(\rho)}\sigma$ et telle que chaque  $\sigma\in\cD(\rho)$ n'apparaît qu'une fois dans $D_0$

\item[(ii)] $D_1(\rho)$ est l'unique représentation de $N$ sur $D_0(\rho)^{I_1}$ qui étend l'action de $I$

\item[(iii)] $r:D_1(\rho)\hookrightarrow D_0(\rho)$ est une injection $I$-équivariante arbitraire.
\end{enumerate}
Puis, avec une telle  $r$ fix\'ee, on obtient par un r\'esultat de \cite[\S9]{BP} une famille de repr\'esentations lisses admissibles $\pi(\rho,r)$ de $G$ telles que \vv

\begin{itemize}
\item[--] $\rsoc_K\pi(\rho,r)=\rsoc_{K}D_0(\rho)$\vspace{1mm}

\item[--]  $D(\rho,r)$ s'injecte dans $(\pi(\rho,r)^{K_1},\pi(\rho,r)^{I_1},\mathrm{can})$ en tant que diagrammes \vspace{1mm}

\item[--] $\pi(\rho,r)$ est engendrée par $D_1(\rho)$ en tant que $G$-représentation.\vspace{2mm}
\end{itemize}
On dira que $\pi(\rho,r)$ est une représentation de $G$ associée à $D(\rho,r)$. Cette proc\'edure est loin d'\^etre canonique et ce n'est pas facile \`a examiner si deux repr\'esentations obtenues ainsi sont isomorphes.  \vv

Fixons un diagramme $D(\rho,r)$ et une représentation
$\pi(\rho,r)$ de $G$ qui lui est associée comme ci-dessus.  On s'intéresse aux
questions suivantes:\vv

(Q1) La représentation $\pi(\rho,r)$ est-elle uniquement déterminée
 par le diagramme  $D(\rho,r)$? Autrement dit, si
$\pi'(\rho,r)$ est une autre représentation de $G$
associée à $D(\rho,r)$, est-ce que l'on a
\[\pi(\rho,r)\cong \pi'(\rho,r)?\]

(Q2) Est-ce que l'on a un isomorphisme de diagrammes:
\[(\pi(\rho,r)^{K_1},\pi(\rho,r)^{I_1},\mathrm{can})\cong (D_0(\rho),D_1(\rho),r)?\]

(Q3) Si $\rho$ est réductible et scindée, on sait que le diagramme
$D(\rho,r)$ se décompose en une somme directe de sous-diagrammes (voir \cite[\S15]{BP}, on pr\'ecise que $f$ est le degr\'e de $F$ sur $\Q_p$)
\[D(\rho,r)=\oplus_{\ell=0}^fD(\rho,r_{\ell})=\oplus_{\ell=0}^f(D_{0,\ell}(\rho),D_{1,\ell}(\rho),r_{\ell}).\]
La représentation $\pi(\rho,r)$ est-elle aussi une
somme directe de sous-représentations $\{\pi_{\ell}, \ 0\leq
\ell\leq f\}$ vérifiant (au moins) la condition
$\rsoc_K(\pi_{\ell})=\rsoc_K(D_{0,\ell}(\rho))$?

\vv

Si $f=1$, alors les réponses \`a (Q1)-(Q3) sont positives, ce qui
a été démontré dans \cite{BP}. Malheureusement, lorsque $f\geq 2$,
les réponses \`a (Q1)-(Q3) sont toutes négatives. En particulier, le diagramme $D(\rho,r)$ \emph{ne suffit pas} pour d\'eterminer une unique repr\'esentation lisse admissible de $G$. Plus précisément, on va démontrer le résultat suivant (cf. théorèmes \ref{theorem-exem-f>3-irr}, \ref{theorem-exem-f>3-red}, \ref{theorem-exem-f=2}):

\begin{theorem}
On conserve les notations précédentes.

(i) Supposons $f=2m+1$ avec $m\geq1$ et $\rho$ irréductible. Il existe deux représentations supersingulières non isomorphes de $G$
associées au même diagramme
$D(\rho,r)$.

(ii) Supposons $f=2$ et $\rho$ irréductible. Il existe une représentation supersingulière $\pi(\rho,r)$
(associée à $D(\rho,r)$) avec $ D_0(\rho)\subsetneq
\pi(\rho,r)^{K_1}$.

(iii) Supposons $f=2m$
avec $m\geq 2$ et $\rho$ réductible scindée. Il existe une représentation $\pi(\rho,r)$
qui n'est pas semi-simple.
\end{theorem}

%
%
%

Introduisons maintenant les principales autres notations de cet article.\vv

Notons $\p:=p\cO$ l'id\'eal maximal de $\cO_F$ et $q:=p^f$ le cardinal du corps r\'esiduel $\cO_F/\p$. On identifie $\cO_F/\p$ avec $\F_q$ qui s'injecte naturellement dans $\bFp$. Pour $\lambda\in\F_q$ on note $[\lambda]$ le repr\'esentant multiplicatif dans $\cO_F$. Pour $a\in\cO_F$ on note $\overline{a}\in\F_q$ la reduction modulo $\p$ de $a$.

Si $n\geq 1$, on note
\[K_n:=\matr{1+\p^n}{\p^n}{\p^n}{1+\p^n},\ \  I_n:=\matr{1+\p^n}{\p^{n-1}}{\p^n}{1+\p^{n}},\]
et $U^+$ (resp. $U^-$) le sous-groupe de $I_1$ des matrices unipotentes sup\'erieures (resp. inf\'erieures). Explicitement, $U^+=\smatr{1}{\cO_F}01$ et $U^-=\smatr{1}0{\p}1$.
  On pose
$\cH\subset K$ le sous-groupe des matrices de la forme
$\smatr{[\lambda]}00{[\mu]}$ avec $\lambda$, $\mu\in\F_q$ et $Z_1:=I_1\cap Z$ le sous-groupe des matrices de la forme $\smatr a00a$ avec $a\in1+\p$.
On désigne par $\Pi$ la matrice $\smatr01{p}0$ (de sorte que $N$ est
engendré par $I$ et $\Pi$).


\vv

Toutes les représentations considérées dans cet article sont sur des
$\bFp$-espaces vectoriels. Pour celles de $G$, on suppose qu'elles admettent
un caractère central. On désigne par $\Rep_G$ (resp. $\Rep_K$,
$\Rep_I$, etc.) la catégorie des représentations lisses de $G$
(resp. $K$, $I$, etc.) sur $\bFp$. \vv

Si $\chi:I\ra \bFp^{\times}$ est un caractère, on note
$\chi^s:=\chi(\Pi\cdot \Pi)$. On pose $\alpha:I\ra\bFp^{\times}$ le
caractère envoyant $\smatr ab{pc}d\in I$ sur
$\overline{a}\overline{d}^{-1}$. Si $M$ est une repr\'esentation lisse de $I$, on note $\Ind_I^KM$ la $\bFp$-repr\'esentation lisse des fonctions $f:K\ra \bFp$ telles que $f(ik)=i\cdot f(k)$ ($i\in I$, $k\in K$) avec action \`a gauche de $K$ par $(kf)(k')=f(k'k)$. Pour $k\in K$ et $m\in M$, on d\'esigne par $[k,m]$ l'\'el\'ement de $\Ind_I^KM$ de support $Ik^{-1}$ et de valeur $m$ en $k^{-1}$. Remarquons que $\Ind_I^K$ est un foncteur exacte de la cat\'egorie $\Rep_I$
dans $\Rep_K$.

Si $\sigma$ est une représentation irréductible de $K$ sur $\bFp$, alors
$\sigma^{I_1}$ est de dimension 1 (\cite[lemme 2]{BL2}) et
on note $\chi_{\sigma}$ le caractère donnant l'action de $I$ sur
$\sigma^{I_1}$ et on note $\sigma^{[s]}$ l'unique représentation
irréductible de $K$ distincte de $\sigma$ et telle que $I$
agit sur $(\sigma^{[s]})^{I_1}$ via $\chi_{\sigma}^s$.\vv

Si $S$ est une représentation lisse d'un groupe profini $H$ (par
exemple, $H=K$ ou $I$), on définit le socle de $M$, noté
$\rsoc(S)$, ou plutôt $\rsoc_H(S)$, comme la plus grande
sous-représentation semi-simple de $S$ et, par récurrence (en posant
$\rsoc(S)=\rsoc^1(S)$), on note $\rsoc^{i+1}(S)$ la
sous-représentation de $S$ contenant $\rsoc^{i}(S)$ telle que
$S_i:=\rsoc^{i+1}(S)/\rsoc^{i}(S)$ soit le socle de $S/\rsoc^{i}(S)$.
Pareillement, on d\'efinit le radical de $S$, noté $\rad(S)$, comme la
plus petite sous-représentation de $S$ telle que le quotient
$S/\rad(S)$ soit semi-simple. On appelle $S/\rad(S)$ le cosocle de $S$, not\'e $\rcosoc(S)$. Le plus petit entier $r$ tel que $\rsoc^r(S)=S$ ou encore $S_r=0$ s'appelle la \emph{longueur de Loewy} de $S$ (cf. \cite[\S1, exercice 2]{Al}) et est not\'e $r(M)$.
Par ailleurs, on écrit la filtration par le socle de $S$ sous la forme:
\[S_0\ \ligne\ S_1\ \ligne\ \cdots\ \ligne\ S_{r(M)-1}.\]

On note
$I(\bQp/F)$ le sous-groupe d'inertie de $\Gal(\bQp/F)$ et on normalise l'injection
du corps de classe local $\iota:F^{\times}\hookrightarrow
\Gal(\bQp/F)^{\mathrm{ab}}$ de telle sorte que les uniformisantes
s'envoient sur les Frobenius géométriques. Grâce à $\iota$, on identifie les
caractères lisses de $F^{\times}$ (resp. de $\cO_F^{\times}$) dans
$\bFp^{\times}$ et les caractères lisses de $\Gal(\bQp/F)$ (resp.
de $I(\bQp/F)$) dans $\bFp^{\times}$. Pour $d\geq 1$,
on note $w_d:I(\bQp/\Q_{p^d})\ra \bFp^{\times}$
le caractère envoyant $g$ sur
$\frac{\overline{g(\sqrt[p^d-1]{p})}}{\sqrt[p^d-1]{p}}\in
\F_{p^d}^{\times}\hookrightarrow \bFp^{\times}$. \vv

Si $\sigma$ et $\tau$ sont deux poids, on dit que $(\sigma,\tau)$ est un couple de
poids de type $(-1,j)$ (resp. $(+1,j)$) si
$\Ext^1_{K}(\tau,\sigma)\neq 0$ et si l'on est dans le cas (a)
(resp. (b)) du corollaire 5.6 (i), \cite{BP}.
Remarquons que ceci implique implicitement $f\geq 2$.\vv

\noindent\textbf{Remerciements}
Ce travail s'est accompli sous la direction
de C. Breuil. Je le remercie chaleureusement pour avoir partagé
avec moi ses idées et ses connaissances et pour toutes ses
remarques. Je voudrais remercier R. Abddellatif, F. Herzig et V. Sécherre pour leurs commentaires et suggestions à la première version. Enfin, je remercie chaleureusement le referee pour ses nombreuses remarques et suggestions constructives.




\section{Combinatoire de
$K$-représentations}\label{subsection-on-K-extensions} Dans ce
paragraphe, on étudie la structure de certaines
$K$-représentations et on généralise \cite[lemme 18.4]{BP}.

\subsection{Rappels et compléments}\label{subsection-principal}
Toute représentation irréductible de $K$ sur $\bFp$ est triviale
sur $K_1$ (cf. \cite[proposition 4]{BL2}), donc est une représentation de $\GL_2(\F_{q})\cong
K/K_1$. Un \emph{poids} (ou \emph{poids de Serre}) est par
définition une telle représentation. Tout poids est,  à isomorphisme près, de la forme (cf.
\cite[proposition 1]{BL2})
\[\bigl(\Sym^{r_0}\bFp^2\otimes(\Sym^{r_1}\bFp^2)^{\Fr}\otimes\cdots
\otimes(\Sym^{r_{f-1}}\bFp^2)^{\Fr^{f-1}}\bigr)\otimes\eta \] où\
les $r_i$ sont des entiers entre $0$ et $p-1$, $\eta$ est un
caractère lisse de $\cO_F^{\times}$ dans $\bFp^{\times}$ vu comme un caract\`ere de $K$ via le d\'eterminant $\det:K\ra \cO_F^{\times}$, et $\Fr:
\GL_2(\F_{q})\ra\GL_2(\F_{q})$ est le Frobenius donné par $\smatr
abcd\mapsto \smatr{a^p}{b^p}{c^p}{d^p}$. On notera
cette représentation $(r_0,\cdots,r_{f-1})\otimes\eta$.

D'autre part, un caract\`ere lisse de $I$ dans $\bFp^{\times}$ est triviale sur $I_1$, donc est un caract\`ere de $\cH\cong I/I_1$. On voit facilement que tout caract\`ere peut s'\'ecrire sous la forme:
\[\matr{a}b{pc}d(\in I)\mapsto \overline{a}^{\sum_{i=0}^{f-1}p^ir_i}\eta(ad)\]
avec $r_i$ et $\eta$ comme pr\'ec\'edement. Par abus de notation, on note ce caract\`ere $(r_0,\cdots,r_{f-1})\otimes\eta$. Notons que $(0,\cdots,0)\otimes\eta\cong(p-1,\cdots,p-1)\otimes\eta$ de telle sorte que l'on demande souvent que les $r_i$ ne soient pas tous $p-1$. 

Si $\sigma=(r_0,\cdots,r_{f-1})\otimes\eta$ est un poids, alors le
caractère $\chi=\chi_{\sigma}$ qui donne l'action de $I$ sur $\sigma^{I_1}$ est exactement le caract\`ere $(r_0,\cdots,r_{f-1})\otimes\eta$ (\cite[\S2]{BP}).
Par réciprocité de Frobenius, le
morphisme $\chi\hookrightarrow \sigma|_I$ correspond à une
surjection $K$-équivariante
$\Ind_I^K\chi\twoheadrightarrow\sigma$. 
De m\^eme, on dispose d'une injection $K$-\'equivariante $\sigma\hookrightarrow\Ind_I^K\chi^s$.
La structure précise de $\Ind_I^{K}\chi$ ou de
$\Ind_I^{K}\chi^s$ a été étudiée dans \cite{Je}. On rappelle
le résultat en suivant la formulation donnée dans \cite{BP}. \vv

 Soient
$(x_0,\cdots,x_{f-1})$ $f$ variables. On définit $\cPx$ comme
l'ensemble des $f$-uplets
$\lambda:=(\lambda_0(x_0),\cdots,\lambda_{f-1}(x_{f-1}))$ où\
$\lambda_i(x_i)\in\Z\pm x_i$ est défini comme suit. Si $f=1$,
$\cP(x_0)=\{x_0,p-1-x_0\}$. Si $f>1$, alors: \vspace{2mm}

 (i)  $\lambda_i(x_i)\in\{x_i,x_i-1,p-2-x_i,p-1-x_i\}$ pour
$i\in\{0,\cdots,f-1\}$ \vspace{2mm}

 (ii)  si $\lambda_i(x_i)\in\{x_i,x_i-1\}$, alors
$\lambda_{i+1}(x_{i+1})\in\{x_{i+1},p-2-x_{i+1}\}$ \vspace{2mm}

 (iii) si $\lambda_i(x_i)\in\{p-2-x_i,p-1-x_i\}$,
alors $\lambda_{i+1}(x_{i+1})\in\{p-1-x_{i+1},x_{i+1}-1\}$
\vspace{2mm}\\
avec les conventions $x_f:=x_0$ et
$\lambda_f(x_f):=\lambda_0(x_0)$.
\\

 Pour $\lambda\in\cPx$, on
définit\vspace{1mm}
\begin{itemize}
\item[]
$e(\lambda):=\dfrac{1}{2}\Bigl(\sum\limits_{i=0}^{f-1}p^i(x_i-\lambda_i(x_i))\Bigr)$
\ \ si $\lambda_{f-1}(x_{f-1})\in\{x_{f-1},x_{f-1}-1\}$

\item[]
$e(\lambda):=\dfrac{1}{2}\Bigl(p^f-1+\sum\limits_{i=0}^{f-1}p^i(x_i-\lambda_i(x_i))\Bigr)$\
\ \ sinon.
\end{itemize}

\vspace{2mm}

Le résultat de \cite{Je} se reformule alors ainsi:

\begin{lemma}\label{lemma-Gamma-PS}
Les sous-quotients irréductibles de $\Ind_{I}^{K}\chi$ ou
$\Ind_{I}^{K}\chi^s$ sont exactement les poids:
\[(\lambda_0(r_0),\cdots,\lambda_{f-1}(r_{f-1}))\otimes{\det}^{e(\lambda)(r_0,\cdots,r_{f-1})}\eta\]
pour $\lambda\in\cPx$ en oubliant les $\lambda$ tels qu'il existe $0\leq i\leq f-1$ vérifiant
$\lambda_i(r_i)<0$.
\end{lemma}
\noindent\textbf{Notation:} ici et dans la suite, si $k\in\Z$ et si $\eta:\cO_F^{\times}\ra\bFp^{\times}$ est un caract\`ere, on \'ecrit $\det^{k}\eta$ au lieu de $\overline{\det}^k\otimes\eta\circ\det$, le caract\`ere de $K$ donn\'e par $g\mapsto (\overline{\det g})^{k}\eta(\det g)$.

\begin{proof}
Voir  \cite[lemme 2.2]{BP}.
\end{proof}

Pour $\lambda\in\cPx$, on définit
\begin{equation}\label{equation-J()}
J(\lambda):=\{i\in\{0,\cdots,f-1\}|\  \lambda_i(x_i)\in\{p-2-x_i,p-1-x_i\}\}.\end{equation}
D'apr\`es le lemme \ref{lemma-Gamma-PS}, si $\tau$ est un sous-quotient irréductible de $\Ind_I^{K}\chi^s$, il existe un unique $\lambda\in\cPx$ tel que $\lambda_i(r_i)\geq0$ pour tout $i\in\{0,\cdots,f-1\}$ et tel que \[\tau=(\lambda_0(r_0),\cdots,\lambda_{f-1}(r_{f-1}))\otimes{\det}^{e(\lambda)(r_0,\cdots,r_{f-1})}\eta.\]
On dit que $\lambda$ est le $f$-uplet correspondant \`a $\tau$ et on pose $J(\tau):=J(\lambda)$.

D'autre part, le lemme
\ref{lemma-Gamma-PS} montre que
$\Ind_I^{K}\chi^s$ est de multiplicité 1.
Par cons\'equent, si $\tau$ est un sous-quotient irr\'eductible de $\Ind_I^{K}\chi^s$, alors il existe une
unique sous-$K$-représentation de $\Ind_I^{K}\chi^s$, notée
$U(\tau)$, dont le cosocle est $\tau$. Pour le voir, soit $\rInj_{\GL_2(\F_q)}\tau$ une enveloppe injective de $\tau$ dans la cat\'egorie $\Rep_{\GL_2(\F_q)}$, vue comme une repr\'esentation lisse de $K$ en faisant agir $K_1$ trivialement. Comme la repr\'esentation $\Ind_{I}^K\chi^s$ est triviale sur $K_1$ et comme $\tau$ y appara\^it avec multiplicit\'e 1, on d\'eduit de \cite[\S5, exercice 2]{Al} que le $\bFp$-espace vectoriel $\Hom_{K}(\rInj_{\GL_2(\F_q)}\tau,\Ind_{I}^K\chi^s)$ est de dimension 1. On prend pour $U(\tau)$ l'image
d'un tel morphisme non nul et la propri\'et\'e que $\rcosoc_K(U(\tau))=\tau$ r\'esulte de \cite[\S6, th\'eor\`eme 6]{Al}.

On peut déterminer la structure  de $U(\tau)$ à
l'aide de $J(\tau)$:
\begin{lemma}\label{lemma-Prin-filtration}
(i) Si $\chi=\chi^s$, alors $\Ind_{I}^K\chi^s=\sigma\oplus\sigma^{[s]}$. En particulier, on a $U(\sigma)=\sigma$ et $U(\sigma^{[s]})=\sigma^{[s]}$.

(ii) Si $\chi\neq \chi^s$, alors les sous-quotients irréductibles de
$U(\tau)$ sont les poids $\tau'\in
\JH(\Ind_I^K\chi^s)$ tels que $J(\tau')\subset J(\tau)$. En particulier, $\rsoc_K(\Ind_I^K\chi^s)=\sigma$.
\end{lemma}
\begin{proof}
Voir \cite[lemme 2.3]{BP} pour (i) et \cite[théorème 2.4]{BP} pour (ii).
\end{proof}

\vv

Si $j$ est un entier entre
$0$ et $f-1$, on note $E_j(\chi)$ l'extension non triviale de
dimension 2:
\begin{equation}\label{equation-define-Ej(chi)}
0\ra \chi\ra E_j(\chi)\ra \chi\alpha^{-p^j}\ra0\end{equation}
où\ l'action de $I$ est donnée, dans une base convenable
$\{v,w\}$, par: si $\smatr ab{pc}d\in I$,
\[\left\{ {\begin{array}{ll}
\matr ab{pc}d v=\chi\Biggl(\matr
ab{pc}d\Biggr)v   \\
\matr ab{pc}dw=\chi\alpha^{-p^j}\Biggl(\matr
ab{pc}d\Biggr)\big(w+(\overline{b/d})^{p^{j}}v\big).
 \\\end{array}}\right.\]
On v\'erifie que cela d\'efinit bien une $I$-extension non triviale de $\chi\alpha^{-p^j}$ par $\chi$. Notons que par d\'efinition $E_j(\chi)$ est triviale sur $K_1$.


Si $M$ est une $I$-représentation, on note $\Pi(M)$ la $I$-représentation définie par
\begin{equation}\label{equation-define-Pi(M)}
h\cdot \Pi(v):=\Pi\bigl((\Pi^{-1}h\Pi)\cdot v\bigr),\ \ h\in I.
\end{equation}
C'est bien d\'efinie puisque $\Pi$ normalise $I$. Notons que $\Pi(\chi)\cong\chi^s$ et $\Pi(\chi\alpha^{-p^j})\cong\chi^s\alpha^{p^j}$, et que $\Pi(E_j(\chi))$ est triviale sur $\Pi K_1\Pi^{-1}=\smatr{1+\p}{\cO_F}{\p^2}{1+\p}$.
\begin{rem}
Si $M'$ est une représentation de $N$ (ou de $G$) et si $M\subset M'$ est
un sous-espace vectoriel stable par $I$, alors
l'espace $\Pi(M)$ est aussi stable par $I$ avec l'action définie
par (\ref{equation-define-Pi(M)}).
\end{rem}


\begin{lemma}\label{lemma-extension-I}
Soit $\chi'$ un caract\`ere lisse de $I$.

(i) $\Ext^1_{I/Z_1}(\chi',\chi)\neq0$ si et seulement si $\chi'=\chi\alpha^{-p^j}$ ou $\chi'=\chi\alpha^{p^j}$ pour un $j\in\{0,\cdots,f-1\}$. Si $\chi'=\chi\alpha^{-p^j}$ (resp. $\chi'=\chi\alpha^{p^j}$) et si $M$ est une $I$-extension non triviale de $\chi'$ par $\chi$, alors $M\cong E_j(\chi)$ (resp. $M\cong \Pi(E_j(\chi^s))$).

(ii) $\Ext^1_{I/K_1}(\chi',\chi)\neq 0$ si et seulement si $\chi'=\chi\alpha^{-p^j}$ pour un $j\in\{0,\cdots,f-1\}$. Si $\chi'=\chi\alpha^{-p^j}$ et si $M$ est une $I/K_1$-extension non triviale de $\chi'$ par $\chi$, alors $M\cong E_j(\chi)|_{U^+\cH}$.
\end{lemma}
\begin{proof}
Les preuves de \cite[propositions 5.2, 5.4 (i), 5.5]{Pa2} qui traitent le cas $f=1$ s'\'etendent au cas g\'en\'eral en remarquant que $\dim_{\bFp}\Hom(\cO_F,\bFp)=f$.
\end{proof}

\vv

Posons $W:=\Ind_I^K\Pi(E_j(\chi))$
et définissons-y les vecteurs suivants (pour $0\leq k\leq q-1$)
\begin{equation}\label{equation-define-F-et-f}
f_k=\summ_{\lambda\in\F_q}\lambda^k\matr{[\lambda]}110[1,\Pi(v)],\
\
F_k=\summ_{\lambda\in\F_q}\lambda^k\matr{[\lambda]}110[1,\Pi(w)],\end{equation}
où\ l'on convient que $0^0:=1$ et $0^{q-1}:=0$.
On a alors les formules suivantes:
\begin{lemma}\label{lemma-Witt-dans-W}
(i) Pour tout $0\leq k\leq q-1$, $F_k$ (resp. $f_k$) est un
vecteur propre de $\cH$ de caractère $\chi\alpha^{-k-p^j}$ (resp.
$\chi\alpha^{-k}$). \vspace{0.8mm}

(ii) Pour tout $0\leq k\leq q-1$,
\[\matr
1p01F_k=F_k+f_k.\]\vspace{0.8mm}

(iii) On a
\[\matr10p1 F_k=\left\{ {\begin{array}{ll} F_k-f_{k+2p^j} &
\mathrm{si\ }  k+2p^j\leq q-1\\

F_k-f_{k+2p^j-(q-1)}& \textrm{si\ } k+2p^j\geq q.
\\\end{array}}\right.\]

(iv) On a
\[\matr{1+p}001F_k=\left\{ {\begin{array}{ll} F_k+f_{k+p^j} &
\mathrm{si\ } k+p^j\leq q-1\\

F_k+f_{k+p^j-(q-1)}& \textrm{si\ } k+p^j\geq q.
\\\end{array}}\right.\]

\end{lemma}
\begin{proof}
(i) Il découle de \cite[lemme 2.5]{BP}.

(ii) Il d\'ecoule de l'égalité
\[\matr1p01\matr{[\lambda]}110=\matr{[\lambda]}110\matr10p1.\]

(iii) Il d\'ecoule du calcul suivant et du fait que $\lambda^{q-1}=1$ pour tout $\lambda\in\F_q^{\times}$: \[\begin{array}{rll}\matr10p1F_k&=&\matr10p1\displaystyle\summ_{\lambda\in\F_q}\lambda^k\matr{[\lambda]}110[1,\Pi(w)]\\
&=&\displaystyle\summ_{\lambda\in\F_q}\lambda^k\matr{[\lambda]}110\matr{1+p[\lambda]}p{-p[\lambda^2]}{1-p[\lambda]}[1,\Pi(w)]\\
&=&\displaystyle\summ_{\lambda\in\F_q}\lambda^k\matr{[\lambda]}110([1,\Pi(w)]-\lambda^{2p^j}[1,\Pi(v)])\\
&=&F_k-\displaystyle\summ_{\lambda\in\F_q}\lambda^{k+2p^j}\matr{[\lambda]}110[1,\Pi(v)].\end{array}\]

(iv) Il d\'ecoule du calcul suivant: \[\begin{array}{rll}\matr{1+p}001F_k&=&\displaystyle\summ_{\lambda\in\F_q}\lambda^k\matr{1+p}001\matr{[\lambda]}110[1,\Pi(w)]\\
&=&\displaystyle\summ_{\lambda\in\F_q}\lambda^k\matr{[\lambda](1+p)}110\matr100{1+p}[1,\Pi(w)]\\
&=&\displaystyle\summ_{\lambda\in\F_q}\lambda^k\matr{[\lambda]}110\matr10{p[\lambda]}1[1,\Pi(w)]\\
&=&F_k+\displaystyle\summ_{\lambda\in\F_q}\lambda^{k+p^j}\matr{[\lambda]}110[1,\Pi(v)].\end{array}\]
\end{proof}

\begin{lemma}\label{lemma-calcul-H}
Soit $R\in W$ une combinaison lin\'eaire des vecteurs $F_k$ et $f_k$ ($0\leq k\leq q-1$), c'est-\`a-dire, $R$ s'\'ecrit sous la forme
\[R=\summ_{0\leq k\leq q-1}a_kF_k+\summ_{0\leq k\leq q-1}b_kf_k\]
avec $a_k,b_k\in\bFp$. Si $a_k\neq 0$ (resp. $b_k\neq0$), alors $F_k$ (resp. $f_k$) appartient \`a $\langle I\cdot R\rangle$, la sous-$I$-repr\'esentation de $W$ engendr\'ee par $R$.
\end{lemma}
\begin{proof}
On peut r\'e\'ecrire $R$ sous la forme $R=\sum_{0\leq k\leq q-2}R_k$ avec
\[R_0=a_{q-1}F_{q-1}+a_{0}F_{0}+b_{p^j}f_{p^j},\ \ \ R_{q-1-p^j}=a_{q-1-p^j}F_{q-1-p^j}+b_{q-1}f_{q-1}+b_{0}f_{0},\]
et $R_{k}=a_{k}F_{k}+b_{k+p^j}f_{k+p^j}$ si $k\notin\{0,q-1-p^j\}$. Alors, par le lemme \ref{lemma-Witt-dans-W} (i), les $R_k$ sont des vecteurs propres de $\cH$ de caract\`ere distinct l'un de l'autre. Le cardinal de $\cH$ \'etant premier \`a $p$, on en d\'eduit que $R_k$ appartient \`a $\langle I\cdot R\rangle$ pour tout $0\leq k\leq q-2$.

Supposons $k\notin\{0,q-1-p^j\}$. Si $a_k\neq 0$, le lemme \ref{lemma-Witt-dans-W} (iv) entra\^ine que $f_{k+p^j}$ appartient \`a $\langle I\cdot R\rangle$ et donc de m\^eme pour $F_k$. Si $a_k=0$ et $b_{k+p^j}\neq 0$, alors $R_k=b_{k+p^j}f_{k+p^j}$ et l'\'enonc\'e est trivial.

Supposons $k=q-1-p^j$. Si $a_{q-1-p^j}\neq 0$, le lemme \ref{lemma-Witt-dans-W} (iv) entra\^ine que $f_{q-1}$ appartient \`a $\langle I\cdot R\rangle$. Puis, on a le calcul suivant:
\begin{equation}\begin{array}{rll}\displaystyle \label{equation-calcul-H}
\matr{1}{1}01f_{q-1}&=&\summ_{\lambda\in\F_q}\lambda^{q-1}\matr{[\lambda]+1}110[1,\Pi(v)]\\
&=&\summ_{\lambda\in\F_q}\lambda^{q-1}\matr{[\lambda+1]}110\matr{1}0{pX_{\lambda}}1[1,\Pi(v)]\\
&=&\summ_{\lambda\in\F_q}(\lambda-1)^{q-1}\matr{[\lambda]}110[1,\Pi(v)]\\
&=&f_{q-1}+\summ_{k'=0}^{q-2}\binom{q-1}{k'}(-1)^{q-1-k'}f_{k'}
\end{array}\end{equation}
o\`u l'on a \'ecrit $[\lambda]+1=[\lambda+1]+pX_{\lambda}$ avec $X_{\lambda}\in \cO_F$ d\'ependant de $\lambda$ et o\`u l'on a utilis\'e le fait que $[1,\Pi(v)]$ est fix\'e par $\smatr{1}0{pX_{\lambda}}1$. Par cons\'equent, le vecteur $\smatr{1}101f_{q-1}-f_{q-1}$ s'\'ecrit d'une combinaison lin\'eaire des vecteurs $f_{k'}$ avec $0\leq k'\leq q-2$ qui sont des vecteurs propres de $\cH$ de caract\`ere distinct l'un de l'autre, d'o\`u $f_0\in \langle I\cdot f_{q-1}\rangle$ et ensuite $F_{q-1-p^j}\in \langle I\cdot R\rangle$. Le calcul (\ref{equation-calcul-H}) montre aussi que $f_{q-1}$ appartient \`a $\langle I\cdot R\rangle$  si $a_{q-1-p^j}=0$ et $b_{q-1}\neq 0$. L'\'enonc\'e est trivial dans le cas o\`u $a_{q-1-p^j}=b_{q-1}=0$ et $b_{0}\neq0$.

Enfin, un argument analogue permet de traiter le cas o\`u $k=0$.
\end{proof}

\begin{prop}\label{prop-W-vecteurs}
Soit $\omega$ un sous-quotient irréductible de
$\Ind_I^K\Pi(\chi\alpha^{-p^j})$. Alors il existe une unique
sous-$K$-représentation de $W$, notée $W_{\omega}$, telle que:

(i) le cosocle de $W_{\omega}$ est isomorphe à $\omega$;

(ii) modulo $\Ind_{I}^K\Pi(\chi)$, l'image de $W_{\omega}$ est
isomorphe à $U(\omega)$,  définie comme étant l'unique
sous-$K$-représentation de $\Ind_I^K\Pi(\chi\alpha^{-p^j})$ de cosocle
$\omega$.
\end{prop}
\begin{rem}
Lorsque $W$ est de multiplicité 1, la proposition \ref{corollary-W-vecteurs} est
\'evidente (par le m\^eme raisonnement que l'existence et l'unicit\'e de $U(\omega)$ plus haut). Cependant, ce n'est pas toujours le cas: par exemple, si $f=2$, $j=0$ et $\chi=\chi_{\sigma}$ avec $\sigma=(1,0)$, alors le poids $(p-2,p-1)\otimes\det$ apparaît dans $W$ avec multiplicité 2.
\end{rem}
\begin{proof}
Soit $\theta\in\cPx$ le $f$-uplet correspondant à $\omega$ par le lemme \ref{lemma-Gamma-PS}, c'est-à-dire,
si l'on écrit $\chi\alpha^{-p^j}=(r_0',\cdots,r_{f-1}')\otimes\eta'$
(en convenant que $r_i'=0$ pour tout $i$ si $\chi\alpha^{-p^j}=(\chi\alpha^{-p^j})^s$), alors $\theta$ est l'unique $f$-uplet tel que
\[\omega=(\theta_0(r_0'),\cdots,\theta_{f-1}(r_{f-1}'))\otimes{\det}^{e(\theta)(r_0',\cdots,r_{f-1}')}\eta'.\]

Considérons le vecteur $F\in W$ défini par
\begin{equation}\label{equation-Ej-define-F}
F=F_{\sum_{i\in
J(\theta)}p^i(p-1-\theta_i(r_i'))}+\epsilon(\omega)\eta'(-1)[1,\Pi(w)]\end{equation}
avec $\epsilon(\omega):=1$ si $\chi\alpha^{-p^j}=(\chi\alpha^{-p^j})^{s}$
et si $\omega$ est de dimension 1, et $\epsilon(\omega):=0$ sinon. Alors $F$ est un vecteur propre de $\cH$ de caract\`ere $\chi_{\omega}$.
On va démontrer que la sous-$K$-repr\'esentation $W_F=\langle K\cdot F\rangle$ de $W$ engendr\'ee par $F$
satisfait aux conditions demandées. D'après
\cite[lemmes 2.6 et 2.7]{BP}, $W_{F}$ satisfait à la condition (ii).

Remarquons que l'ensemble des sous-$K$-représentations de $W$
vérifiant (i) et (ii) est non vide. Soit $W'$ une telle
représentation. Alors
$W'$ contient un vecteur $F'$ qui s'écrit sous la forme
\[F'=F+af\]
o\`u
$a\in\bFp$ et $f\in\Ind_I^K\Pi(\chi)$ est un vecteur propre
de $\cH$ du même caractère que $F$. Le lemme \ref{lemma-calcul-H} impliqu'alors que $F\in W'$ et donc $W_F\subset W'$ (si $\epsilon(\omega)=1$, on utilise le fait que $[1,\Pi(w)]$ est fix\'e par $\smatr{1+p}001$
pour d\'eduire l'appartenance de $f$ \`a $W'$). Cela permet de conclure que $W'=W_F$, car sinon on aurait $W_F\subset \rad_K(W')$ et donc $\omega$, \'etant le cosocle de $W_F$ et de $W'$, appara\^itrait dans $W'/(W'\cap \Ind_I^K\Pi(\chi))\hookrightarrow \Ind_{I}^K\Pi(\chi\alpha^{-p^j})$ avec multiplicit\'e $\geq 2$, ce qui contredirait le lemme \ref{lemma-Gamma-PS}. Le corollaire s'en d\'eduit.
\end{proof}
\vv

Si $0\leq k\leq q-1$, on \'ecrit (de mani\`ere unique) $k=\sum_{0\leq i\leq f-1}p^ik_i$ avec $0\leq k_i\leq p-1$ pour tout $0\leq i\leq f-1$. Rappelons que $U^+=\smatr{1}{\cO_F}01\subset I$.

\begin{prop}\label{prop-U+-action}
Fixons $k$ un entier entre $0$ et $q-1$.

(i) Dans $\Ind_I^K\Pi(\chi)$, la sous-$U^+$-repr\'esentation $\langle U^+\cdot f_k\rangle$ engendr\'ee par $f_k$ est stable par $I$. Elle a une $\bFp$-base form\'ee des vecteurs suivants
\[\Bigl\{f_{\sum_{0\leq i\leq f-1}{p^ik_i'}},\ 0\leq k_i'\leq k_i\Bigr\}.\]

(ii) Dans $W$, la sous-$U^+$-repr\'esentation $\langle U^+\cdot F_k\rangle$ engendr\'ee par $F_k$ contient tous les vecteurs
\[\Bigl\{f_{\sum_{0\leq i\leq f-1}p^ik_i'}, \ 0\leq k_{j-1}'\leq p-1\ \mathrm{et}\ 0\leq k_i'\leq k_i\ \mathrm{si}\ i\neq j-1\Bigr\}.\]
\end{prop}
\begin{proof}
(i) D'abord, comme $f_k$ est fix\'e par $K_1$ et est un vecteur propre de $\cH$, la stabilit\'e de $\langle U^+\cdot f_k\rangle$ d\'ecoule de la d\'ecomposition $I=U^+K_1\cH$. De plus, on est ramen\'e pour le deuxi\`eme \'enonc\'e \`a examiner l'action sur $f_k$ des matrices de la forme $\smatr1{[\mu]}01$ avec $\mu\in\F_q$.
En rempla\c{c}ant $q-1$ par $k$ et $\smatr{1}101$ par $\smatr{1}{[\mu]}01$ dans le calcul (\ref{equation-calcul-H}), on obtient:
\[\matr1{[\mu]}01f_{k}=\summ_{k'=0}^{k}\binom{k}{k'}(-\mu)^{k-k'}f_{k'},\]
et l'\'enonc\'e s'en d\'eduit par le lemme \ref{lemma-calcul-H} puisque $\binom{k}{k'}\neq 0$ si et seulement si $k'=\sum_{0\leq i\leq f-1}p^ik_i'$ avec $0\leq k_i'\leq k_i$.

(ii) Cette preuve est extraite de \cite[lemme 18.4, cas -1]{BP}. D'abord,
d'apr\`es (i), $\langle U^+\cdot F_{k}\rangle $ contient un vecteur qui s'\'ecrit sous la forme
\[F_{\sum_{i\neq j-1}p^ik_i}+bf_{\sum_{i\neq j-1}p^ik_i+p^{j}}\]
avec $b\in\bFp$. Puis, en utilisant le fait que \[X_{\lambda-1}\equiv \summ_{s=1}^{p-1}\frac{\binom{p}{s}}{p}(\lambda-1)^{p^{-1}s}
\equiv-\summ_{s=1}^{p-1}\frac{\binom{p}{s}}{p}\lambda^{p^{-1}s}(-1)^{p^{-1}(p-s)} \mod p^2\]
dont la premi\`ere \'egalit\'e vient de la loi d'addition dans $\cO_F$ (cf. \cite[\S II.6]{Se0}) et la deuxi\`eme est un exercice en combinatoire, on obtient par un calcul analogue \`a (\ref{equation-calcul-H}):
\[\matr1101f_{\sum_{i\neq j-1}p^ik_i+p^{j}}=\summ_{\lambda\in\F_q}(\lambda-1)^{\sum_{i\neq j-1}p^ik_i+p^{j}}\matr{[\lambda]}110[1,\Pi(v)]\]
et
\begin{equation}
\begin{array}{rll} 
\matr{1}{1}01F_{\sum_{i\neq j-1}p^ik_i}
&=&\summ_{\lambda\in\F_q}(\lambda-1)^{\sum_{i\neq j-1}p^ik_i}\matr{[\lambda]}110\matr{1}0{pX_{\lambda}}1[1,\Pi(w)]
\\
&=&\summ_{\lambda\in\F_q}(\lambda-1)^{\sum_{i\neq j-1}p^ik_i}\matr{[\lambda]}110[1,\Pi(w)]\label{equation-U^+-action}\\
&& \
+\summ_{\lambda\in\F_q}(\lambda-1)^{\sum_{i\neq j-1}p^ik_i}\Biggl(-\summ_{s=1}^{p-1}\frac{\binom{p}{s}}{p}\lambda^{p^{j-1}s}(-1)^{p^{j-1}(p-s)}\Biggr)\matr{[\lambda]}110[1,\Pi(v)].
\end{array}\end{equation}
On constate que le vecteur $f_{\sum_{i\neq j-1}p^ik_i+p^{j-1}(p-1)}$ n'appara\^it que dans la deuxi\`eme \'equation (\ref{equation-U^+-action}) et le coefficient est $1$. D'ailleurs, on v\'erifie que tout vecteur $F_{k'}$ qui appara\^it dans (\ref{equation-U^+-action}) a un caract\`ere propre de $\cH$ diff\'erent de $\chi\alpha^{-\sum_{i\neq j-1}p^ik_i-p^{j-1}(p-1)}$. Cela permet de conclure que le vecteur $f_{\sum_{i\neq j-1}p^ik_i+p^{j-1}(p-1)}$ appartient \`a $\langle U^+\cdot F_k\rangle$, sauf dans le cas o\`u $k_i=p-1$ pour tout $i\neq j-1$ auquel cas le calcul (\ref{equation-calcul-H}) s'applique. Le r\'esultat s'en d\'eduit par (i).
\end{proof}

La proposition \ref{prop-U+-action} permet d'am\'eliorer le r\'esultat de \cite[corollaire 5.6]{BP} qui traite des $K$-extensions entre deux poids. Rappelons que $\sigma=(r_0,\cdots,r_{f-1})\otimes\eta$ et $\chi=\chi_{\sigma}$.

\begin{prop}\label{prop-extension-K}
Supposons que $f\geq 2$.

(i) Si $r_i\leq p-2$ pour tout $0\leq i\leq f-1$, alors $\Ext^1_{K/Z_1}(\sigma,\sigma)=0$.

(ii) Supposons que $r_j\leq p-3$ et posons $\tau=(r_0,\cdots,r_{j}+2,\cdots,r_{f-1})\otimes{\det}^{-p^j}\eta$. Si $r_{j-1}\leq p-2$, alors $\Ext^1_K(\tau,\sigma)=0$.

(iii) Supposons que $r_j\geq 2$ et posons $\tau=(r_0,\cdots,r_j-2,\cdots,r_{f-1})\otimes{\det}^{p^j}\eta$. Si $r_{j-1}\leq p-2$, alors $\Ext^1_K(\tau,\sigma)=0$.
\end{prop}
\begin{proof}
(i) Par l'absurde, soit
\[0\ra V_1\ra V\ra V_2\ra 0\]
une $K$-extension non scind\'ee, triviale sur $Z_1$ et telle que $V_1\cong V_2\cong\sigma$.
Soit $w\in V$ un vecteur propre de $\cH$ tel que son image $\overline{w}$ dans $V_2$ soit fix\'ee par $I_1$. Alors $w$ est de caract\`ere propre $\chi$ engendrant $V$ sous l'action de $K$, et on a les \'enonc\'es suivants sur $w$:
\begin{enumerate}
\item[--] $w$ n'est pas fix\'e par $I_1$: sinon on aurait une surjection $K$-\'equivariante $\Ind_I^K\chi\twoheadrightarrow V$ ce qui est impossible par le lemme \ref{lemma-Gamma-PS};\vv

\item[--] $w$ est fix\'e par $U^+$: ceci d\'ecoule du lemme \ref{lemma-extension-I} (ii) parce qu'aucun des caract\`eres $\{\chi\alpha^{p^i},\ 0\leq i\leq f-1\}$ n'appara\^it dans $\sigma|_I$;\vv

\item[--] $w$ n'est pas fix\'e par $U^-$ et le vecteur $v=\smatr{1}0p1w-w$ qui appartient \`a $V_1$ est non nul: par ce qui pr\'ec\`ede, le premier \'enonc\'e d\'ecoule du fait que le groupe $I_1$ est engendr\'e par $U^+$, $U^-$ et $Z_1$, et le deuxi\`eme en est une cons\'equence: si $\smatr{1}0p1w=w$, alors pour tout $\mu\in\F_q^{\times}$ \[\matr{1}0{p[\mu]}1w=\matr{[\mu^{-1}]}001\matr{1}0p1\matr{[\mu]}001w=w,\]
    et puis $w$ serait fix\'e par $U^-$ parce que les matrices $\smatr10{p[\mu]}1$ l'engendrent topologiquement;\vv

\item[--] $w$ n'est pas fix\'e par $\smatr{1+\p}001$:
la d\'ecomposition d'Iwahori (combin\'ee avec le fait que $V$ est fix\'e par $K_2$)
\begin{equation}\label{equation-Iwahori}
\matr 1{b}01 \matr{1}0p1=
\matr{1}{0}{\frac{p}{1+pb}}{1}\matr{1+pb}{0}{0}{\frac{1}{1+pb}}
\matr{1}{\frac{b}{1+pb}}{0}{1}, \ \ \forall b\in\cO_F
\end{equation}
montre que, si $v$ est fix\'e par $\smatr{1+\p}001$, il est aussi fix\'e par $U^+$ et donc appartient \`a $V_1^{I_1}$ (\cite[lemme 2]{BL2}) de telle sorte que $\bFp w\oplus\bFp v$ fournisse une $I/Z_1$-extension de $\chi$ par $\chi$, ce qui contredit le lemme \ref{lemma-extension-I}.
\end{enumerate}

Soit maintenant $u\in V_1^{I_1}$ un vecteur non nul (unique \`a scalaire pr\`es). Puisque $w$ est fix\'e par $U^+$, un calcul simple montre qu'il est de m\^eme pour tout vecteur $u_{a}=\smatr{1+pa}001w-w$ (o\`u $a\in\cO_F$), d'o\`u $u_a\in\bFp u$ car $u_a\in V_1$ par le choix de $w$. Autrement dit, $\bFp u\oplus\bFp w$ induit une extension non triviale de $\smatr{1+\p}001$-repr\'esentations
\[0\ra \bFp u\ra *\ra \bFp w\ra0.\]
Or, une telle extension correspond \`a un caract\`ere lisse de $1+\p$ dans $\bFp^{\times}$, c'est-\`a-dire, il existe $j\in\{0,\cdots,f-1\}$ tel que (quitte à multiplier $u$ par un scalaire)
\[\matr{1+pa}001w=w+\frac{1}{2}\overline{a}^{p^j}u,\ \ \forall a\in\cO_F.\]
Par cons\'equent, la d\'ecomposition (\ref{equation-Iwahori}) donne
\[\matr{1}b{0}1v=v+\overline{b}^{p^j}u,\]
et ensuite le lemme \ref{lemma-extension-I} implique que $v$ est un vecteur propre de $\cH$ de caract\`ere $\chi\alpha^{-p^j}$ et que $\bFp u\oplus\bFp v$ fournit une $I$-repr\'esentation isomorphe \`a $E_j(\chi)$.

Somme toute, on a montr\'e que la $I$-repr\'esentation $M_w=\langle I\cdot w\rangle$ est de dimension $3$ dont une base est form\'ee par $\{u,v,w\}$  et qu'elle est isomorphe \`a
\[0\ra E_j(\chi)\ra M_{w}\ra \chi\ra0.\]
Consid\'erons la surjection $K$-\'equivariante $\phi:\Ind_I^KM_{w}\twoheadrightarrow V$ induite par r\'eciprocit\'e de Frobenius. Alors l'image de $\Ind_I^K(\bFp u\oplus\bFp v)$ dans $V$ ne peut pas \^etre r\'eduit \`a 0 ni \'egaler \`a $V$, donc est forc\'ement isomorphe \`a $V_1$ de telle sorte que l'on obtienne une surjection
\[\overline{\phi}:\Ind_I^K\bFp \overline{w}\cong\Ind_I^K\chi\twoheadrightarrow V_2\]
o\`u $\overline{w}$ d\'esigne l'image de $w$ dans $M_{w}/(\bFp u\oplus \bFp v)$.
En utilisant \cite[lemme 2.7]{BP} et l'hypoth\`ese que $r_i\neq p-1$ pour tout $0\leq i\leq f-1$, on trouve que le noyau $\ker(\overline{\phi})$ contient le vecteur
\[\summ_{\lambda\in\F_q}\lambda^{\sum_{i\neq j-1}p^i(p-1)}\matr{[\lambda]}110[1,\overline{w}],\]
ou de mani\`ere \'equivalente, $\ker(\phi)$ contient un vecteur propre de $\cH$ s'\'ecrivant sous la forme
\[R=\summ_{\lambda\in\F_q}\lambda^{\sum_{i\neq j-1}p^i(p-1)}\bigl([1,w]+a[1,u]\bigr)+f_{v}\]
pour certains $a\in\bFp$ et $f_v\in\Ind_I^K\bFp v$. Un calcul comme dans la preuve de la proposition \ref{prop-U+-action} (ii) montre que $R$ engendre le vecteur $\sum_{\lambda\in\F_q}\lambda^{q-1}\smatr{[\lambda]}110[1,v]$, qui engendre de plus $\Ind_I^K(\bFp u\oplus\bFp v)$ par le lemme \ref{lemma-Ind(Ej)} (ii) ci-apr\`es. Cela donne une contradiction et termine la d\'emonstration.\vv

(ii) On proc\`ede par la m\^eme ligne que (i). Plus pr\'ecis\'ement, soient $V$ une $K$-extension non scind\'ee de $\tau$ par $\sigma$ et $w\in V$ un vecteur propre de $\cH$ tel que son image dans $\tau$ soit fix\'ee par $I_1$. Alors on montre que $w$ est fix\'e par $U^+$ et par $\smatr{1+\p}001$, et que le vecteur $v=\smatr{1}0p1w-w$ qui appartient \`a $\sigma$ est non nul fix\'e par $I_1$.
Autrement dit, le sous-espace vectoriel $\bFp v\oplus\bFp w$ de $V$ est stable par $I$ isomorphe \`a  $\Pi(E_j(\chi^s))$, et on en d\'eduit une surjection $K$-\'equivariante $W=\Ind_I^K\Pi(E_j(\chi^s))\twoheadrightarrow V$. Apr\`es, comme dans (i), on montre que $W$ n'admet pas de quotient isomorphe \`a $V$ en utilisant la proposition \ref{prop-U+-action} et l'hypoth\`ese que $r_{j-1}\neq p-1$.

(iii) Il d\'ecoule de (ii) par dualit\'e. Soit $V$ une $K$-extension de $\tau$ par $\sigma$. Si l'on note $V^{*}$ la repr\'esentation duale de $V$ d\'efinie par $V^*:=\Hom_{\bFp}(V,\bFp)$ avec l'action usuelle de $K$ (voir \cite[\S6]{Al}), alors $V^{*}$ est une extension de $\sigma^*$ par $\tau^*$. De plus, on v\'erifie que $\sigma*$ (resp. $\tau^*$) est isomorphe \`a $\sigma$ (resp. $\tau$) \emph{\`a torsion pr\`es}, donc $V^*$ se scinde d'apr\`es (ii) et par cons\'equent $V$ se scinde aussi.\end{proof}
\vv

On termine cette section par le lemme suivant qui traite la repr\'esentation induite $\Ind_I^KE_j(\chi)$. On \'ecrit $\chi\alpha^{-p^j}=(r_0',\cdots,r_{f-1}')\otimes\eta'$ comme dans le lemme \ref{lemma-pour-rj'}.
\begin{lemma}\label{lemma-Ind(Ej)}
Supposons que $r_j'\leq p-2$.

(i) Dans $\Ind_I^KE_j(\chi)$, le vecteur
$\sum_{\lambda\in\F_q}\smatr{[\lambda]}110[1,w]$
est fix\'e par $I_1$ et engendre une sous-$K$-repr\'esentation irr\'eductible isomorphe au poids $(p-1-r_0',\cdots,p-1-r_{f-1}')\otimes{\det}^{\sum_{0\leq i\leq f-1}p^ir_i'}\eta'$.

(ii) Le vecteur $\sum_{\lambda\in\F_q}\lambda^{q-1}\smatr{[\lambda]}110[1,w]$ engendre enti\`erement $\Ind_I^KE_j(\chi)$ sous l'action de $K$.
\end{lemma}
\begin{proof}
\'Ecrivons $R_0=\sum_{\lambda\in\F_q}\smatr{[\lambda]}110[1,w]$, $R_{q-1}=\sum_{\lambda\in\F_q}\lambda^{q-1}\smatr{[\lambda]}110[1,w]$, et $\sigma'=(p-1-r_0',\cdots,p-1-r_{f-1}')\otimes{\det}^{\sum_{0\leq i\leq f-1}p^ir_i'}\eta'$.

(i) Puisque la $I$-repr\'esentation $E_j(\chi)$ est triviale sur $K_1$, l'induite $\Ind_I^KE_j(\chi)$ l'est aussi. Donc, pour le premier \'enonc\'e, il suffit de v\'erifier que le vecteur $R_0$ est fix\'e par $U^+$, ou encore par les matrices $\smatr{1}{[\mu]}01$ avec $\mu\in\F_q$, ce qui d\'ecoule du fait que $w$ est fix\'e par $U^-=\smatr{1}0{\p}1$:
\[\matr1{[\mu]}01R_0=\summ_{\lambda\in\F_q}\matr{[\mu]+[\lambda]}110[1,w]=
\summ_{\lambda\in\F_q}\matr{[\mu+\lambda]}110[1,w]=R_0.\]

La condition $r_j'\leq p-2$ assure que le poids $\sigma'$ admet un $I$-quotient isomorphe \`a $E_j(\chi)$ (par un raisonnement analogue au lemme \ref{lemma-I-extension-r}) (ii). Donc on obtient par r\'eciprocit\'e de Frobenius une injection $K$-\'equivariante $\sigma'\hookrightarrow \Ind_I^KE_j(\chi)$ qui envoie $\sigma'^{I_1}$ vers $\bFp R_0$, d'o\`u le r\'esultat.

(ii) En utilisant le fait que $w$ est fix\'e par $U^-$, le calcul (\ref{equation-calcul-H}) montre que la $K$-repr\'esentation $\langle K\cdot R_{q-1}\rangle$ contient $R_0$ et puis le vecteur $[1,w]$ puisque
\[\matr0110[1,w]=R_0-R_{q-1}.\]
Le r\'esultat s'en d\'eduit car $[1,w]$ engendre enti\`erement $\Ind_I^KE_j(\chi)$.
\end{proof}

\subsection{La structure de $W_{\omega}$}

On conserve les notations du \S\ref{subsection-principal}. En particulier, $\chi=\chi_{\sigma}$ est le caract\`ere donnant l'action de $I$ sur l'espace des $I_1$-invariants de $\sigma=(r_0,\cdots,r_{f-1})\otimes\eta$, $E_j(\chi)$ est la repr\'esentation de $I$ d\'efinie par (\ref{equation-define-Ej(chi)}), et $W=\Ind_I^K\Pi(E_j(\chi))$.

Si $\omega$ est un sous-quotient irr\'eductible de $\Ind_I^K\Pi(\chi\alpha^{-p^j})$, on a d\'efini une unique sous-$K$-repr\'esentation $W_{\omega}$ de $W$ v\'erifiant certaines conditions (cf. proposition \ref{prop-W-vecteurs}). Dans \cite{BP}, la structure de $W_{\omega}$ pour tout
sous-quotient irréductible <<spécial>> $\omega$ de
$\Ind_I^K\Pi(\chi\alpha^{-p^j})$ a été déterminée (cf. \cite[définition 17.2 et lemme 18.4]{BP}). Nous
allons généraliser ce résultat.
Commençons par un lemme
facile.

\begin{lemma}\label{lemma-pour-rj'}
Écrivons $\chi\alpha^{-p^j}$ sous la forme
$({r_0'},\cdots,r_{f-1}')\otimes\eta'$ pour des $r_i'$ et $\eta'$
convenables (uniquement déterminés en demandant que, si
$\chi\alpha^{-p^j}=(\chi\alpha^{-p^j})^s$, alors $r_i'=0$ pour
tout $0\leq i\leq f-1$).

(i) Si $r_j\geq 2$, alors $r_j'=r_j-2$ et $r_i'=r_i$ pour tout
$i\neq j$.

(ii) Si $r_j=1$, alors $r_j'\in \{p-2,p-1\}$. De plus, $r_j'=p-2$ si et
seulement si $r_i=0$ pour tout $i\neq j$, auquel cas on a
$\chi\alpha^{-p^j}=\chi^s$; si $r_j'=p-1$, soit $j'$ le premier
indice suivant $j$ tel que $r_{j'}\geq 1$, alors
$(r_j',\cdots,r_{j'}')=(p-1,p-1,\cdots,p-1,r_{j'}-1)$.

(iii) Si $r_j=0$, alors $r_j'\in\{p-2,p-3\}$. De plus, $r_j'=p-3$ si et
seulement si $r_i=0$ pour tout $i\neq j$, i.e.
$\chi=(0,\cdots,0)\otimes\eta$; si $r_j'=p-2$, soit $j'$ le
premier indice suivant $j$ tel que $r_{j'}\geq 1$, alors
$(r_j',\cdots,r_{j'}')=(p-2,p-1,\cdots,p-1,r_{j'}-1)$.
\end{lemma}
\begin{proof}
\'El\'ementaire.
\end{proof}

%
%


Rappelons que, si $\lambda\in\cPx$, alors $J(\lambda)$ est d\'efini par (\ref{equation-J()}):
\[J(\lambda):=\{i\in\{0,\cdots,f-1\}|\  \lambda_i(x_i)\in\{p-2-x_i,p-1-x_i\}\}.\]

\begin{lemma}\label{lemma-Breuil}
Supposons $f\geq 2$ et $\chi\neq \chi^s$.  Soient $\omega$ (resp. $\tau$) un sous-quotient irréductible
de $\Ind_I^K\Pi(\chi\alpha^{-p^j})$ (resp. $\Ind_I^K\Pi(\chi)$) et $\theta\in\cPx$ (resp. $\lambda$) le
$f$-uplet correspondant.  \'Ecrivons $\chi\alpha^{-p^j}$ sous la forme
$({r_0'},\cdots,r_{f-1}')\otimes\eta'$ comme dans le lemme
\ref{lemma-pour-rj'}.

(i) Si $r_j\geq 2$, alors $U(\tau)\subset W_{\omega}$ si
et seulement si $J(\lambda)\subset J(\theta)\cup\{j-1\}$.

(ii) Si $r_j\leq 1$, alors $U(\tau)\subset
W_{\omega}$ si et seulement si $J(\lambda)\subset J(\theta)\cup
J'\cup\{j,j-1\}$, o\`u
\[J':=\{i\in\{0,\cdots,f-1\}|\
r_i'=p-1, r_i=0\}.\]
\end{lemma}

\begin{proof}
(i) 
Démontrons d'abord la direction $\Longleftarrow$. Puisque $\chi\neq\chi^s$ par hypoth\`ese, on d\'eduit de
\cite[lemme 2.7]{BP} et du lemme \ref{lemma-calcul-H} que, d'une part la représentation $U(\tau)$
est engendrée par le vecteur
\[f_{\sum_{i\in J(\lambda)}p^i(p-1-\lambda_i(r_i))},\]
et d'autre part $W_{\omega}$ contient le vecteur suivant selon le cas (cf. (\ref{equation-Ej-define-F}) de la preuve de la proposition \ref{prop-W-vecteurs} dans le premier cas):
\[\left\{ {\begin{array}{ll} F_0+\eta'(-1)[1,\Pi(w)] &
\mathrm{si\ } \chi\alpha^{-p^j}=(\chi\alpha^{-p^j})^s\ \mathrm{et}\ \dim_{\bFp}\omega=1  \\

F_{\sum_{i\in J(\theta)}p^i(p-1)}& \textrm{sinon}.
\\\end{array}}\right.\]
Dans le deuxi\`eme cas, la proposition \ref{prop-U+-action} implique que $W_{\omega}$ contient le vecteur $f_{\sum_{i\in J(\theta)\cup\{j-1\}}p^i(p-1)}$ et donc $U(\tau)\subset W_{\omega}$ d\`es que $J(\lambda)\subset
J(\theta)\cup\{j-1\}$. Dans le premier cas, en utilisant le fait que $[1,\Pi(w)]$ est fix\'e par $\smatr1101$), on obtient le calcul analogue \`a (\ref{equation-U^+-action}):
\[\begin{array}{rll}\Biggl(\matr1101 -1\Biggr)\bigl(F_0+\eta'(-1)[1,\Pi(w)]\bigr)&=&\Biggl(\matr1101-1\Biggr)F_0\\
&=&\summ_{s=1}^{p-1}\Bigl(-\frac{\binom{p}{s}}{p}(-1)^{p^{j-1}(p-s)}\Bigr)f_{p^{j-1}s};\end{array}\]
par cons\'equent, le lemme \ref{lemma-calcul-H} et la proposition \ref{prop-U+-action} (i) impliquent que $W_{\omega}$ contient tous les vecteurs $f_{p^{j-1}s}$ avec $0\leq s\leq p-1$, donc contient $U(\tau)$ si $J(\lambda)\subset \{j-1\}$ (notons que $J(\theta)=\emptyset$ dans ce cas).

D\'emontrons la direction $\Longrightarrow$: c'est-\`a-dire, v\'erifions
que $\tau$
n'apparaît pas dans $W_{\omega}$ lorsque
$J(\lambda)\nsubseteq J(\theta)\cup\{j-1\}$. Clairement, une condition n\'ecessaire pour que $U(\tau)$ soit contenu dans $W_{\omega}$ est qu'il existe $\tau'\in \JH\bigl(\Ind_{I}^K\Pi(\chi)\bigr)$ et $\omega'\in \JH\bigl(\Ind_I^K\Pi(\chi\alpha^{-p^j})\bigr)$ tels que \[U(\tau)\subset U(\tau')\subset W_{\omega'}\subset W_{\omega}\] et tels que $\Ext^1_K(\omega',\tau')\neq0$. On est donc ramen\'e \`a v\'erifier que $\Ext^1_K(\omega,\tau)=0$ dès que
$J(\lambda)\nsubseteq J(\theta)\cup\{j-1\}$. Avec cette hypothèse,
soit $k\neq j-1$ un indice tel que $k\notin J(\theta)$ et $k\in
J(\lambda)$. On a alors (puisque $r_j\geq 2$):
\[\omega=(\theta_0(r_0),\cdots,\theta_j(r_j-2),\cdots,\theta_{f-1}(r_{f-1}))\otimes
{\det}^{e(\theta)(r_0,\cdots,r_j-2,\cdots,r_{f-1})+p^j}\eta\] et
\[\tau=(\lambda_{0}(r_0),\cdots,\lambda_{j}(r_{j}),\cdots,\lambda_{f-1}(r_{f-1}))\otimes
{\det}^{e(\lambda)(r_0,\cdots,r_{f-1})}\eta.\]  Comme
$\theta_j(x_j-2)\notin\{\lambda_j(x_j), p-2-\lambda_j(x_j)\}$ et comme
$k\neq j-1$ et $\theta_k(x_k)\neq\lambda_{k}(x_k)$ (par hypothèse),
on voit que $\omega$ et $\tau(\lambda)$ ne satisfont \`a aucune condition list\'ee dans \cite[corollaire 5.5 (ii)]{BP}, d'où le résultat.


\vv

(ii)
Supposons maintenant $r_j\leq 1$. Puisque $\chi\neq \chi^s$, le lemme \ref{lemma-pour-rj'} implique que $r_j'\in\{p-1,p-2\}$. D'autre part, on a par d\'efinition $r_i'=p-1$ pour $i\in J'$. On d\'eduit donc de \cite[lemme 2.7]{BP} et du lemme \ref{lemma-calcul-H} que $W_{\omega}$ contient le vecteur
\[F_{\sum_{i\in J(\theta)\backslash(J'\cup\{j\})}p^i(p-1)+\sum_{i\in J'}p^i(p-2)+p^{j}(p-3)},\]
qui engendre d'apr\`es la proposition \ref{prop-U+-action} (ii) le vecteur (notons que $j-1\notin J'$ puisque $\chi\neq \chi^s$)
\[f_{\sum_{i\in J(\theta)\backslash(J'\cup\{j,j-1\})}p^i(p-1)+\sum_{i\in J'}p^i(p-2)+p^{j}(p-3)+p^{j-1}(p-1)}.\]
Puis, par notre hypoth\`ese que $p\geq 5$ et $\chi\neq \chi^s$, on voit que ce dernier vecteur engendre $U(\tau)$ (sous l'action de $K$) pour tout $\tau$ v\'erifiant
$J(\lambda)\subset J(\theta)\cup J'\cup\{j-1,j\}$.

Pour l'autre direction, on peut procéder comme dans (i).
\end{proof}

\begin{cor}\label{corollaire-W-chi=chi^s}
Supposons $f\geq 2$ et $\sigma=(0,\cdots,0)\otimes\eta$ de telle sorte que $\chi=\chi^s$. Soient $\omega$ un sous-quotient irr\'eductible de $\Ind_I^K\Pi(\chi\alpha^{-p^j})$ et $\theta\in\cPx$ le $f$-uplet correspondant.
Alors on a toujours $\sigma^{[s]}\subset W_{\omega}$, et on a $\sigma\subset W_{\omega}$ si et seulement si $J(\omega)\cup \{j-1\}=\{0,\cdots,f-1\}$.
\end{cor}
\begin{proof}
Il découle du lemme \ref{lemma-Breuil} par dualité. En effet, si l'on note $W^*$ la repr\'esentation duale de $W$ comme dans la preuve de la proposition \ref{prop-extension-K} (iii), alors $W^*$ est isomorphe \`a l'extension non scind\'ee
\[0\ra \Ind_I^K\Pi(\tilde{\chi})\ra W^{*}\ra\Ind_I^K\Pi(\tilde{\chi}\alpha^{-p^j})\ra0\]
avec $\tilde{\chi}=(0,\cdots,0,2_j,0,\cdots,0)\otimes\tilde{\eta}$ (ici, $2_j$ signale que $2$ est \`a la place $j$) pour certain caract\`ere $\tilde{\eta}$ de $I$. On peut donc appliquer le lemme \ref{lemma-Breuil} (i) et l'\'enonc\'e s'en d\'eduit en utilisant le fait que $U(\tau)\subset W_{\omega}$ si et seulement si $U({\omega}^*)\subset W_{\tau^*}$ en regardant $\omega^*$ (resp. $\tau^*$) comme un sous-quotient irr\'eductible de $\Ind_I^K\Pi(\tilde{\chi})$ (resp. $\Ind_{I}^K\Pi(\tilde{\chi}\alpha^{-p^j})$). 
\end{proof}
Le cas o\`u $f=1$ est plus simple et a \'et\'e traité dans \cite{Br1, Br2}:
\begin{lemma}\label{lemma-Breuil-f=1}
Si $f=1$ et si $\omega$ est un sous-quotient irréductible de
$\Ind_I^K\Pi(\chi\alpha^{-1})$, alors
$\Ind_I^K\Pi(\chi)\subseteq W_{\omega}$.
\end{lemma}
\begin{proof}
Le même argument que celui du lemme \ref{lemma-Breuil} traite le cas,
où $\chi\neq\chi^s$ et
$\chi\alpha^{-1}\neq(\chi\alpha^{-1})^s$, i.e.
$r_0\notin\{0,2\}$. Si $r_0=2$, l'énoncé se déduit de
\cite[lemme 11.8]{Br2}, qui traite en fait le cas où $r_0\geq 2$. Enfin, le
cas où $r_0=0$ découle par dualité du cas où $r_0=2$ (cf. la preuve du corollaire \ref{corollaire-W-chi=chi^s}).
\end{proof}

On voit que le lemme \ref{lemma-Breuil-f=1} est compatible avec le lemme \ref{lemma-Breuil} et le corollaire \ref{corollaire-W-chi=chi^s}.

\subsection{Généralisation}

Dans cette section, on généralise le lemme \ref{lemma-Breuil} à un cas
plus général.

\begin{lemma}\label{lemma-I-extension-r}
Soient $\chi$ un caractère lisse de $I$ et $j$ un
entier entre $0$ et $f-1$.

(i)  Si $\lon$ est un entier entre $0$ et $p-1$, alors il existe
une unique $I$-représentation (\`a isomorphisme pr\`es), notée $E_j(\chi,\lon)$, qui est
triviale sur $K_1$, unisérielle de dimension $s+1$ et dont la filtration par le socle est de la forme
\[\chi\ \ligne\ \chi\alpha^{-p^j}\ \ligne\ \cdots\ \ligne\ \chi\alpha^{-p^j(s-1)}\ \ligne\ \chi\alpha^{-p^js}.\]
(Voir \cite[\S4, proposition 5]{Al} pour la notion <<unis\'erielle>>.)


(ii) Le poids $\sigma=(r_0,\cdots,r_{f-1})\otimes\eta$ contient, pour tout $0\leq j\leq f-1$, une sous-$I$-représentation isomorphe à
$E_j(\chi_{\sigma},r_j)$.
\end{lemma}
\begin{proof}
(i) Pour l'existence de $E_j(\chi,s)$, il suffit de prendre la sous-$I$-représentation
de $\Ind_I^K\Pi(\chi)$ engendrée par le vecteur
\[f_{p^js}=\summ_{\lambda\in\F_q}\lambda^{p^js}\matr{[\lambda]}{1}{1}{0}[1,\Pi(v)],\]
où\ $v$ est un vecteur de base de $\chi$. Le fait que cette repr\'esentation satisfait aux conditions demand\'ees d\'ecoule de la proposition \ref{prop-U+-action} (i) et du lemme \ref{lemma-Witt-dans-W} (i).

Par le lemme \ref{lemma-extension-I} (ii), si $0\leq i< s\leq p-1$, on a
\[\dim_{\bFp}\Ext^1_{I/K_1}(\chi\alpha^{-p^js},\chi\alpha^{-p^ji})=\left\{ {\begin{array}{ll}
1 & i=s-1  \\
0& i\neq s-1.
 \\\end{array}}\right.\]
L'unicité de la représentation $E_j(\chi,\lon)$ s'en d\'eduit
par r\'ecurrence sur $s$.

(ii) 
Le poids $\sigma$ s'injecte dans $\Ind_I^K\Pi(\chi_{\sigma})$ et, par \cite[lemme 2.7]{BP}, contient le vecteur $f_{p^jr_j}$ si $\{r_i, \ i\neq j\}$ ne sont pas tous nuls (resp. $f_{p^jr_j}+\eta(-1)[1,\Pi(v)]$ si $r_i=0$ pour tout $i\neq j$). L'\'enonc\'e s'en d\'eduit d'apr\`es la proposition \ref{prop-U+-action} (i) (resp. combin\'ee avec le fait que $[1,\Pi(v)]$ est fix\'e par $I_1$).
\end{proof}

Supposons $f\geq 2$ jusqu'à la fin de cette section. Si $s$ est un entier entre $0$ et $p-2$ et si $\chi$, $\chi'$ sont deux caractères lisses de $I$ tels que
$\chi\alpha^{-p^{j-1}(\lon+1)}=\chi'\alpha^{-p^{j}}$, on note
$E_{j-1}(\chi,\chi',\lon+1)$ 
l'unique $I$-représentation qui rend exacte la suite suivante (notons que $E_{j-1}(\chi,s+1)$ est bien d\'efinie puisque $0\leq s\leq p-2$)
\[0\ra E_{j-1}(\chi,\chi',\lon+1)\ra E_{j-1}(\chi,\lon+1)\oplus E_{j}(\chi')\overset{h}{\ra}
\chi\alpha^{-p^{j-1}(\lon+1)}\ra0\] o\`u $h=h_1-h_2$, avec $h_1$ (resp.
$h_2$) la projection naturelle de $E_{j-1}(\chi,\lon+1)$ (resp.
$E_{j}(\chi')$) sur $\chi\alpha^{-p^{j-1}(\lon+1)}$. De mani\`ere explicite, la filtration par le socle de $E_{j-1}(\chi,\chi',\lon+1)$ est la suivante:
\[\begin{array}{rccccccccccccccc}\chi&\ligne&\chi\alpha^{-p^{j-1}}&\ligne&\cdots&\cdots&\ligne
&\chi\alpha^{-p^{j-1}(\lon+1)}=\chi'\alpha^{-p^j}\\
&&&&&&&\mid\\
&&&&&&&\chi'.\end{array}\]

\begin{rem}\label{remark-E=multi-1}
(i) On v\'erifie que $E_{j-1}(\chi,\chi',s+1)$ est de dimension $s+3$ et de multiplicit\'e 1.

(ii) Par d\'efinition, en tant que repr\'esentation de $U^+$, la longueur de Loewy de $E_{j-1}(\chi,\chi',s+1)$ est \'egale \`a $s+2$, c'est-\`a-dire, $s+2$ est le plus petit entier $r$ v\'erifiant la propri\'et\'e \[\rsoc_I^{r}\bigl(E_{j-1}(\chi,\chi',s+1)\bigr)=E_{j-1}(\chi,\chi',s+1).\]
\end{rem}

La raison pour laquelle on introduit cette notation se voit dans le lemme suivant: 
\begin{lemma}\label{lemma-E(deux-char)}
Soient $(\sigma,\tau)$ un couple de poids de type $(+1,j)$ et $V$
une $K$-extension non triviale de $\tau$ par $\sigma$:
\[0\ra\sigma\ra V\ra \tau\ra0.\]
Supposons $\sigma=(r_0,\cdots,r_{f-1})\otimes\eta$. Alors $V$ contient une unique sous-$I$-représentation isomorphe à
$E_{j-1}(\chi_{\sigma},\chi_{\tau},r_{j-1}+1)$.
\end{lemma}
\begin{proof}
Remarquons d'abord que l'on a automatiquement $0\leq r_{j-1}\leq p-2$ par \cite[corollaire 5.6 (i)]{BP}.

Le lemme \ref{lemma-Prin-filtration} montre que $V$ est un quotient de $\Ind_I^K\chi_{\tau}$. Donc, si $v$ est un vecteur de base de $\chi$, c'est une cons\'equence de la proposition \ref{prop-U+-action} (i) que l'image dans $V$ du vecteur
\[\summ_{\lambda\in\F_q}\lambda^{\sum_{i\neq j-1}p^i(p-1-r_{i})+p^{j-1}(r_{j-1}+1)}\matr{[\lambda]}110[1,\Pi(v)]\in\Ind_{I}^K\Pi(\chi_{\tau})\]
est non nulle d'apr\`es \cite[lemme 2.7]{BP} et
engendre une sous-$I$-repr\'esentation de $V$ isomorphe
à $E_{j-1}(\chi_{\sigma},\chi_{\tau},r_{j-1}+1)$.
\end{proof}

Conservons les notations du lemme \ref{lemma-E(deux-char)}. On va \'etudier
la structure de la représentation induite
$W=\Ind_I^K\Pi(E_{j-1}(\chi_{\sigma},\chi_{\tau},r_{j-1}+1))$.\vv

Soient $\psi$ un caract\`ere de $I$ apparaissant dans $E_{j-1}(\chi_{\sigma},\chi_{\tau},r_{j-1}+1)$ et $v$ un vecteur propre de $\cH$ de caract\`ere $\psi$. Notons $E_{v}$ la
sous-$I$-représentation engendrée par $v$ et $\rad_I(E_v)$ son
radical. Puisque $E_{j-1}(\chi_{\sigma},\chi_{\tau},r_{j-1}+1)$ est de
multiplicité 1, $v$ est unique à scalaire près de sorte que $E_v$
et $\rad_I(E_v)$ sont bien définis.
\begin{lemma} \label{lemma-W-vecteurs-k}
Soit $\omega$ un
sous-quotient irréductible de $\Ind_I^K\Pi(\psi)$. Il existe une unique sous-$K$-représentation de $W$, notée
$W_{\omega}$, telle que:
\begin{itemize}
\item[--] le cosocle de $W_{\omega}$ est isomorphe à $\omega$;

\item[--] modulo $\Ind_I^K\Pi(\rad_I(E_v))$, l'image de $W_{\omega}$
est isomorphe à $U(\omega)$, l'unique sous-représentation de
$\Ind_{I}^K\Pi(\psi)$ admettant $\omega$ comme cosocle.
\end{itemize}
\end{lemma}

\begin{proof}
Analogue à  celle de la proposition \ref{prop-W-vecteurs}. Explicitement, $W_{\omega}$ est engendr\'ee par le vecteur
\begin{equation}\label{equation-def-vecteur-W}
\summ_{\lambda\in\F_q}\lambda^{\sum_{i\in J(\lambda)}p^i(p-1-\lambda_i(r_i(\psi)))}\matr{[\lambda]}110\Pi(v_{\psi})+\epsilon(\psi)\eta(\psi)(-1)[1,\Pi(v_{\psi})]
\end{equation}
o\`u
\begin{enumerate}
\item[--] $v_{\psi}$ est un vecteur non nul propre de $\cH$ de caract\`ere $\psi$ (unique \`a scalaire pr\`es)\vspace{1mm}

\item[--] $\epsilon(\psi)=1$ si $\psi=\psi^s$ et $\omega$ est de dimension 1, et $\epsilon(\psi)=0$ sinon\vspace{1mm}

\item[--] $\psi$ s'\'ecrit sous la forme $(r_0(\psi),\cdots,r_{f-1}(\psi))\otimes\eta(\psi)$ \vspace{1mm}

\item[--] $\lambda\in\cPx$ est l'unique $f$-uplet correspondant \`a $\omega$.
\end{enumerate}
\end{proof}

On suppose que $r_{j}\geq 1$ dans le corollaire
suivant. Cette hypoth\`ese assure que $\psi\neq \psi^s$ pour tout caract\`ere $\psi$ distinct de $\chi\alpha^{-p^{j-1}(r_{j-1}+1)}$ apparaissant dans $E_{j-1}(\chi_{\sigma},\chi_{\tau},r_{j-1}+1)$.

\begin{cor}\label{corollary-W-U(tau)}
Soient $\omega$ un poids apparaissant dans
$\Ind_I^K\Pi(\chi_{\sigma}\alpha^{-p^{j-1}(r_{j-1}+1)})$ et $\sigma'$ (resp.
$\tau'$) un poids apparaissant dans $\Ind_I^K\Pi(\chi_{\sigma})$ (resp.
$\Ind_I^K\Pi(\chi_{\tau})$). Alors

(i) $U(\sigma')\subset W_{\omega}$ si et seulement si
$J(\sigma')\subset J(\omega)\cup\{j-2,j-1\}$;

(ii) $U(\tau')\subset W_{\omega}$ si et seulement si
$J(\tau')\subset J(\omega)\cup\{j-1\}$.
\end{cor}
\begin{proof}
On a explicitement
\[\tau=(r_0,\cdots,p-2-r_{j-1},r_j+1,\cdots,r_{f-1})\otimes{\det}^{p^{j-1}(r_{j-1}+1)-p^j}\eta,\]
donc (ii) est une conséquence du
lemme \ref{lemma-Breuil} (i) puisque $r_j+1\geq 2$ (notons que ce lemme est applicable car $\chi_{\tau}\neq \chi_{\tau}^s$).

Pour (i), on peut travailler dans
$\Ind_I^K\Pi(E_{j-1}(\chi_{\sigma},r_{j-1}+1))$, i.e. modulo
$\Ind_I^K\Pi(\chi_{\tau})$. Posons $t=[r_{j-1}/2]$ de telle sorte que l'on ait
une suite exacte
\[0\ra E_{j-1}(\chi_{\sigma},t)\ra E_{j-1}(\chi_{\sigma},r_{j-1}+1)\ra E_{j-1}(\chi_{\sigma}\alpha^{-p^{j-1}(t+1)},r_{j-1}-t)\ra 0\]
et que le cosocle de $E_{j-1}(\chi_{\sigma},t)$, isomorphe à
$\chi_{\sigma}\alpha^{-p^{j-1}t}$, s'écrive sous la forme
\[(r_0',\cdots,r_{j-1}',\cdots,r'_{f-1} )\otimes\eta'\]
avec $r_{j-1}'\leq 1$, $r_{j}'=r_j-1$ et $r_i'=r_i$ pour tout $i\notin\{j-1,j\}$. Le lemme \ref{lemma-Breuil} impliqu'alors que,
si $\omega'$ est un sous-quotient irréductible de
$\Ind_I^K\Pi(\chi_{\sigma}\alpha^{-p^{j-1}(t+1)})$, $W_{\omega'}$
contient $U(\sigma')$ si et seulement si $J(\sigma')\subseteq
J(\omega')\cup\{j-1,j-2\}$. D'autre part, si
$\overline{W_{\omega}}$ désigne l'image de $W_{\omega}$ dans
$\Ind_I^K\Pi(E_{j-1}(\chi\alpha^{-p^{j-1}(t+1)},r_{j-1}-t))$, le
même lemme implique que $\overline{W_{\omega}}$ contient
$U(\omega')$ si et seulement si $J(\omega')\subseteq
J(\omega)\cup\{j-2\}$. En somme, le corollaire est prouv\'e.
\end{proof}


\begin{exem}\label{exemple-1}
Conservons les notations du corollaire \ref{corollary-W-U(tau)} et prenons $\omega$ comme le sous-quotient irr\'eductible $\Ind_I^K\Pi(\chi_{\sigma}\alpha^{-p^{j-1}(r_{j-1}+1)})$ tel que $J(\omega)=\emptyset$. Supposons de plus $r_{j-2}\neq p-1$. Alors la filtration par le socle de $W_{\omega}$ est de la forme:
\[\begin{array}{cccccccccccccccccccc}
\sigma^{(0)}_{\{j-1\}}&\ligne&\sigma^{(0)}_{J}
&\cdots&\ligne&\sigma^{(t)}_{\{j-1\}}&\ligne&\sigma^{(t)}_{J}&\ligne
&\sigma^{(t+1)}_{\emptyset}&\ligne&\sigma^{(t+1)}_{\{j-2\}}
&\cdots&\ligne&\omega\\
\mid&&\mid&&&\mid&&\mid&&&&&&&\mid\\
\sigma^{(0)}_{\emptyset}&\ligne&\sigma_{\{j-2\}}^{(0)}&\cdots&\ligne
&\sigma_{\emptyset}^{(t)}&\ligne&\sigma^{(t)}_{\{j-2\}}
&&&&&&&\tau_{\{j-1\}}\\
&&&&&&&&&&&&&&\mid\\
&&&&&&&&&&&&&&\tau_{\emptyset},\end{array}\]

o\`u
\begin{enumerate}
\item[--] $t=[r_{j-1}/2]$ et $J=\{j-1,j-2\}$; \vspace{1mm}

\item[--] si $0\leq i\leq r_{j-1}$ et si $*\in\bigl\{\emptyset,\{j-1\},\{j-2\},J\bigr\}$, alors $\sigma^{(i)}_{*}$ d\'esigne l'unique  sous-quotient irr\'eductible de $\Ind_I^K\Pi(\chi_{\sigma}\alpha^{-p^{j-1}i})$ tel que $J(\sigma^{(i)}_{*})=*$;\vspace{1mm}

\item[--] idem pour $\tau_{\emptyset}$ et $\tau_{\{j-1\}}$ mais en tant que sous-quotients irr\'eductibles de $\Ind_I^K\Pi(\chi_{\tau})$; en fait, on a $\tau_{\emptyset}=\tau$ et
$\tau_{\{j-1\}}=\sigma$;\vspace{1mm}

\item[--] on doit oublier $\sigma_{\{j-2\}}^{(t)}$ dans ce diagramme si $r_{j-1}$ est pair.
\end{enumerate}

De mani\`ere explicite, lorsque $f\geq 3$, on a:
\begin{enumerate}
\item[--] si $0\leq i\leq t$, $\sigma_{\emptyset}^{(i)}=(r_0,\cdots,r_{j-2},r_{j-1}-2i,r_{j},\cdots,r_{f-1})\otimes{\det}^{p^{j-1}i}\eta$\vspace{1mm}

\item[--] si $0\leq i\leq t$, $\sigma_{\{j-2\}}^{(i)}=(r_0\cdots,p-2-r_{j-2},r_{j-1}-1-2i,r_{j},\cdots,r_{f-1})\otimes
    {\det}^{p^{j-1}i+p^{j-2}(r_{j-2}+1)}\eta$ \vspace{1mm}


\item[--] si $0\leq i\leq t$, $\sigma_{\{j-1\}}^{(i)}=(r_0,\cdots,r_{j-2},p-2-r_{j-1}+2i,r_{j}-1,\cdots,r_{f-1})\otimes
    {\det}^{p^{j-1}(r_{j-1}+1-i)}\eta$\vspace{1mm}

\item[--] si $0\leq i\leq t$, $\sigma_{J}^{(i)}=(r_0,\cdots,p-2-r_{j-2},p-1-r_{j-1}+2i,r_j-1,\cdots,r_{f-1})\otimes
    {\det}^{p^{j-1}(r_{j-1}-i)+p^{j-2}(r_{j-2}+1)}\eta$\vspace{1mm}

\item[--] si $t+1\leq i\leq r_{j-1}$, $\sigma_{\emptyset}^{(i)}=(r_0,\cdots,r_{j-2},p+r_{j-1}-2i,r_{j}-1,\cdots,r_{f-1})\otimes{\det}^{p^{j-1}i}\eta$ \vspace{1mm}

\item[--] si $t+1\leq i\leq r_{j-1}$, $\sigma_{\{j-2\}}^{(i)}=(r_0,\cdots,p-2-r_{j-2},p+r_{j-1}-1-2i,r_{j}-1,\cdots,r_{f-1})
    \otimes{\det}^{p^{j-1}i+p^{j-2}(r_{j-2}+1)}\eta$.
\end{enumerate}
On en d\'eduit que $\sigma_{\{j-1\}}^{(0)}\cong\omega$ et pour tout $1\leq i\leq t$: \[\sigma_{\{j-1\}}^{(i)}\cong\sigma_{\emptyset}^{(r_{j-1}-i)},\ \ \sigma_{J}^{(i-1)}\cong\sigma_{\{j-2\}}^{(r_{j-1}+1-i)}.\]


On peut de m\^eme d\'ecrire explicitement les poids $\sigma^{(i)}_{*}$ lorsque $f=2$. 
\end{exem}

\section{Combinatoire des diagrammes de Diamond génériques}

Dans ce paragraphe, on rappelle la construction des diagrammes de
Diamond (\cite{BP}) et on en fournit des propriétés combinatoires.

\subsection{Diagrammes de Diamond génériques}\label{subsection-PoidsDiamond}
Les $\bFp$-représentations continues de dimension 2 de
$\Gal(\overline{\Q}_p/F)$ sont classifiées à l'aide des
caractères fondamentaux de Serre $\omega_d$ ($d\geq 1$).
De manière explicite:
\begin{prop}\label{prop-galoisside}
Soit $\rho:\Gal(\bQp/F)\ra\GL_2(\bFp)$ une  représentation continue. Alors $\rho$ est de l'une des formes suivantes:

(i) $\rho$ est réductible et
\[\rho|_{I(\bQp/F)}\cong\matr{\omega_f^{\sum_{i=0}^{f-1}p^i(r_i+1)}}*0{1}\otimes\eta\]
où\ $\eta$ est un caractère lisse de $I(\bQp/F)$ qui se prolonge à
$\Gal(\bQp/F)$ et les $r_i$ sont des entiers entre $-1$ et $p-2$
tels que $(r_0,\cdots,r_{f-1})\neq(p-2,\cdots,p-2)$.

(ii) $\rho$ est irréductible et
\[\rho|_{I(\bQp/F)}\cong\matr{\omega_{2f}^{\sum_{i=0}^{f-1}p^i(r_i+1)}}00
{\omega_{2f}^{q\sum_{i=0}^{f-1}p^i(r_i+1)}}\otimes\eta\] où\
$\eta$ est un caractère lisse de $I(\bQp/F)$ qui se prolonge à
$\Gal(\bQp/F)$ et les $r_i$ sont des entiers tels que $0\leq
r_0\leq p-1$, $-1\leq r_i\leq p-2$ pour $i>0$ et
$(r_0,\cdots,r_{f-1})\neq (p-2,\cdots,p-2)$.
\end{prop}
\begin{proof}
Voir \cite[corollaire 2.9]{Br2}.
\end{proof}

\begin{defn}(\cite[\S11]{BP})\label{definition-generique}
Conservons les notations de la proposition
\ref{prop-galoisside}. La représentation $\rho$ est dite
\emph{générique} si $0\leq r_i\leq p-3$, et
$(r_0,\cdots,r_{f-1})\notin\{(0,\cdots,0),(p-3,\cdots,p-3)\}$ dans
le cas réductible, ou si $1\leq r_0\leq p-2$ et $0\leq r_i\leq p-3$
pour $i>0$ dans le cas irréductible.
\end{defn}

À la représentation $\rho$, on peut associer un ensemble $\cD(\rho)$ de poids,
appelés \emph{poids de Diamond} (\cite{BDJ}). Si
$\rho$ est de plus semi-simple et générique, on
peut décrire $\cD(\rho)$ comme suit (\cite[\S11]{BP}). \vv

Soit $(x_0,\cdots,x_{f-1})$ $f$ variables. On commence par définir
 deux ensembles $\cRDx$ et $\cIDx$ de $f$-uplets
$\lambda=(\lambda_0(x_0),\cdots,\lambda_{f-1}(x_{f-1}))$ où
$\lambda_i(x_i)\in\Z\pm x_i$. On convient que $x_f=x_0$ et
$\lambda_f(x_f)=\lambda_0(x_0)$ dans ce qui suit.\vv
\begin{itemize}
\item[--] Si $f=1$, $\cRDx:=\{x_0,p-3-x_0\}$ et
$\cIDx:=\{x_0,p-1-x_0\}$.\vv

\item[--] Si $f>1$, $\cRDx$ est l'ensemble des $\lambda$ tels que:\vv
\begin{itemize}
\item[(i)] $\lambda_i(x_i)\in\{x_i,x_i+1,p-2-x_i,p-3-x_i\}$ pour
tout $i\in\{0,\cdots,f-1\}$\vspace{1mm}

\item[(ii)] si $\lambda_i(x_i)\in\{x_i,x_i+1\}$, alors
$\lambda_{i+1}(x_{i+1})\in\{x_{i+1},p-2-x_{i+1}\}$\vspace{1mm}

\item[(iii)] si $\lambda_i(x_i)\in\{p-2-x_i,p-3-x_i\}$, alors
$\lambda_{i+1}(x_{i+1})\in\{p-3-x_{i+1},x_{i+1}+1\}$\vv\vv
\end{itemize}

et $\cIDx$ est l'ensemble des $\lambda$ tels que:\vv
\begin{itemize}
\item[(i)] $\lambda_0(x_0)\in\{x_0,x_0-1,p-2-x_0,p-1-x_0\}$ et
$\lambda_i(x_i)\in\{x_i,x_i+1,p-2-x_i,p-3-x_i\}$ si
$i>0$\vspace{1mm}

\item[(ii)] si $i>0$ et $\lambda_i(x_i)\in\{x_i,x_i+1\}$ (resp.
$\lambda_0(x_0)\in\{x_0,x_0-1\}$), alors
$\lambda_{i+1}(x_{i+1})\in\{x_{i+1},p-2-x_{i+1}\}$\vspace{1mm}

\item[(iii)] si $0<i<f-1$ et
$\lambda_i(x_i)\in\{p-2-x_i,p-3-x_i\}$, alors
$\lambda_{i+1}(x_{i+1})\in\{p-3-x_{i+1},x_{i+1}+1\}$\vspace{1mm}

\item[(iv)] si $\lambda_0(x_0)\in\{p-1-x_0,p-2-x_0\}$, alors
$\lambda_1(x_1)\in \{p-3-x_1,x_1+1\}$\vspace{1mm}

\item[(v)] si
$\lambda_{f-1}(x_{f-1})\in\{p-2-x_{f-1},p-3-x_{f-1}\}$, alors
$\lambda_0(x_0)\in\{p-1-x_0,x_0-1\}$.\vspace{4mm}
\end{itemize}
\end{itemize}

 Pour $\lambda\in \cRDx$ ou $\lambda\in\cIDx$, on
pose\vspace{1mm}
\begin{itemize}
\item[]
$e(\lambda):=\frac{1}{2}\big(\summ_{i=0}^{f-1}p^i(x_i-\lambda_i(x_i))\big)$
\ \ si $\lambda_{f-1}(x_{f-1})\in\{x_{f-1},x_{f-1}+1\}$

\item[]
$e(\lambda):=\frac{1}{2}\big(p^f-1+\summ_{i=0}^{f-1}p^i(x_i-\lambda_i(x_i))\big)$
\ \ sinon.
\end{itemize}\vv

\begin{lemma}\label{lemma-galois-serre-irred}
Supposons $\rho$ semi-simple et générique. Si $\rho$
vérifie (i) (resp. (ii)) de la proposition
\ref{prop-galoisside}, alors $\cD(\rho)$ est exactement
l'ensemble des poids
\[(\lambda_0(r_0),\cdots,\lambda_{f-1}(x_{f-1}))\otimes{\det}^{e(\lambda)(r_0,\cdots,r_{f-1})}\eta\]
pour $\lambda\in\cRDx$ (resp. $\lambda\in\cIDx$).
\end{lemma}
\begin{proof}
Voir  \cite[lemmes 11.2, 11.4]{BP}.
\end{proof}

On peut
identifier l'ensemble $\cRDx$ (resp. $\cIDx$) avec l'ensemble des
sous-ensembles $\cS$ de $\{0,\cdots,f-1\}$ comme suit:\vspace{1mm}
\begin{itemize}
\item[--] pour $\lambda\in\cRDx$, on pose $i\in\cS$ si et
seulement si $\lambda_i(x_i)\in\{p-3-x_i,x_i+1\}$\vspace{1mm}

\item[--] pour $\lambda\in\cIDx$, on pose $0\in\cS$ si et
seulement si $\lambda_0(x_0)\in\{p-1-x_0,x_0-1\}$ et, si $i>0$ on
pose $i\in \cS$ si et seulement si
$\lambda_i(x_i)\in\{p-3-x_i,x_i+1\}$.
\end{itemize}\vspace{1mm}
Si $\lambda\in\cRDx$ (resp. $\cIDx$), et si $\sigma$ est le poids de Diamond correspondant (par le lemme \ref{lemma-galois-serre-irred}), on note $\cS_{\lambda}$ le
sous-ensemble de $\{0,\cdots,f-1\}$ qui lui est associé ci-dessus et on pose:
\begin{equation}\label{equation-ell}\ell(\sigma):=|S_{\lambda}|.
\end{equation}
\begin{lemma}\label{combination-corollary}
 Si $\lambda,\lambda'\in\cRDx$ (resp. $\cIDx$) sont tels qu'il existe
$j\in\{0,\cdots,f-1\}$ vérifiant:
\[\cS_{\lambda}\cap(\{0,\cdots,f-1\}\backslash\{j-1,j\})=\cS_{\lambda'}
\cap(\{0,\cdots,f-1\}\backslash\{j-1,j\}),\] alors
$\lambda_i(x_i)=\lambda_i'(x_i)$ pour tout $i\notin\{j-2,j-1,j\}$ (où l'on identifie $k\in\{j-2,j-1\}$ avec $k+f$ si $k<0$).
\end{lemma}
\begin{proof}
C'est clair à partir de la définition.
\end{proof}
\vspace{2mm}



Rappelons (\cite[\S9]{BP}) qu'un \emph{diagramme} est
par définition un triplet $(D_0,D_1,r)$ où\ $D_0$ est une
représentation lisse de $KZ$, $D_1$ est une représentation lisse
de $N$ et $r: D_1\ra D_0$ est un morphisme $IZ$-équivariant. On
définit des morphismes entre deux diagrammes de manière évidente et
on note $\DIAG$ la catégorie qui en résulte.

Supposons que la représentation $\rho$ est telle que $p\in Z$ agisse trivialement sur $\det(\rho)$. On lui associe une famille de
diagrammes $D=(D_0(\rho),D_1(\rho),r)$ comme suit
(\cite[\S13]{BP}):\vv

(i) $D_0(\rho)$ est la plus grande représentation de $\GL_2(\F_q)$ sur
$\bFp$ (pour l'inclusion) telle que $\rsoc_{\GL_2(\F_q)}
D_0(\rho)=\oplus_{\sigma\in\cD(\rho)}\sigma$ et telle que chaque
$\sigma\in\cD(\rho)$ n'apparaisse qu'une seule fois dans
$D_0(\rho)$; on la voit comme repr\'esentation de $KZ$ en faisant agir $K_1$ et $p\in Z$ trivialement; \vspace{1mm}

(ii) $D_1(\rho)$ est l'unique représentation de $N$ sur
$D_0(\rho)^{I_1}$ qui étend l'action de $IZ$;\vspace{1mm}

(iii) $r:D_1(\rho) \hookrightarrow D_0(\rho)$ est une injection
$IZ$-équivariante arbitraire.\vspace{1mm}

Remarquons qu'en
général, il y a un nombre infini d'injections $r$ à isomorphisme
près.\vv

Nous aurons besoin de la description explicite de $D_0(\rho)$.\vv

Soit $(y_0,\cdots,y_{f-1})$ $f$ variables. On définit
un ensemble $\cIy$ de $f$-uplets
$\mu:=(\mu_0(y_0),\cdots,\mu_{f-1}(y_{f-1}))$ avec
$\mu_i(y_i)\in\Z\pm y_i$ comme suit. Si $f=1$,
$\mu_0\in\{y_0,p-1-y_0,p-3-y_0\}$. Si $f>1$,
alors:\vspace{2mm}

(i) $\mu_i(y_i)\in\{y_i,y_i-1,y_i+1,p-2-y_i,p-3-y_i,p-1-y_i\}$
pour $i\in\{0,\cdots,f-1\}$\vspace{2mm}

(ii) si $\mu_i(y_i)\in\{y_i,y_i-1,y_i+1\}$, alors
$\mu_{i+1}(y_{i+1})\in\{y_{i+1},p-2-y_{i+1}\}$\vspace{2mm}

(iii) si $\mu_i(y_i)\in\{p-2-y_i,p-3-y_i,p-1-y_i\}$, alors
$\mu_{i+1}(y_{i+1})\in\{y_{i+1}-1,y_{i+1}+1,p-3-y_{i+1},p-1-y_{i+1}\}$
\vspace{2mm}\\
avec les conventions $y_f:=y_0$ et
$\mu_{f}(y_f):=\mu_0(y_0)$.\vv

\begin{defn}(\cite[\S4]{BP})\label{definition-compatible}
Soient $\mu, \mu'\in\cIy$. On dit que $\mu$ et
$\mu'$ sont \emph{compatibles} si, pour tout $0\leq i\leq f-1$, $\mu_i(y_i)$ et $\mu'_i(y_i)$ appartiennent \`a la fois soit \`a $\{y_i,p-2-y_i,y_i+1,p-3-y_i\}$, soit \`a $\{y_i,p-2-y_i,y_i-1,p-1-y_i\}$.
\end{defn}\vv

Pour $\lambda\in \cRDx$ (resp. $\cIDx$), on définit
l'élément $\mu_{\lambda}\in\cIy$ comme suit:\vspace{1mm}

(i) $\mu_{\lambda,i}(y_i):=p-1-y_i$ si
$\lambda_i(x_i)\in\{p-3-x_i,x_i\}$ (resp. si $i>0$ ou, si $i=0$ et
$\lambda_0(x_0)\in\{p-2-x_0,x_0-1\}$)\vspace{1mm}

(ii) $\mu_{\lambda,i}(y_i):=p-3-y_i$ si
$\lambda_i(x_i)\in\{p-2-x_i,x_i+1\}$ (resp. si $i>0$ ou, si $i=0$
et $\lambda_0(x_0)\in\{p-1-x_0,x_0\}$). \vspace{2mm}
\begin{theorem}\label{theorem-BP-D0(rho)}
On conserve les notations précédentes.

(i) $D_0(\rho)$ se décompose en une somme directe:
\[D_0(\rho)=\bigoplus_{\sigma\in\cD(\rho)}D_{0,\sigma}(\rho)\]
avec $\rsoc_{K}D_{0,\sigma}(\rho)\cong \sigma$.

(ii) Soient $\sigma\in\cD(\rho)$ et $\lambda$ le $f$-uplet correpondant. Alors les sous-quotients irréductibles de $D_{0,\sigma}(\rho)$ sont
exactement les poids:
\[\bigl(\mu_0(\lambda_0(r_0)),\cdots,\mu_{f-1}(\lambda_{f-1}(r_{f-1}))\bigr)\otimes
{\det}^{e(\mu\circ\lambda)(r_0,\cdots,r_{f-1})}\eta\] pour
$\mu\in\cIy$ tels que $\mu$ et $\mu_{\lambda}$ soient compatibles
(cf. définition \ref{definition-compatible}) en oubliant les
poids tels qu'il existe $i$ vérifiant $\mu_i(\lambda_i(r_i))<0$ ou
$\mu_i(\lambda_i(r_i))>p-1$. En particulier,
$D_0(\rho)$ est de multiplicité 1.

(iii) Supposons que $\rho$ est réductible. Alors, en tant que diagramme, $D(\rho,r)$ se
décompose en une somme directe de la forme (on pr\'ecise que $f$ est le degr\'e de $F$ sur $\Q_p$): \[D(\rho,r)=\bigoplus_{\ell=0}^{f}D_{\ell}(\rho,r)=\bigoplus_{\ell=0}^{f}\bigl(\oplus_{\ell(\sigma)
=\ell}D_{0,\sigma}(\rho),\oplus_{\ell(\sigma)=\ell}D_{1,\sigma}(\rho),r\bigr).\]
\end{theorem}
\begin{proof}
Voir \cite[proposition 13.4]{BP} pour (i), [\emph{ibid.}, théorème 14.8] pour (ii), et [\emph{ibid.}, théorème 15.4 (ii)] pour
(iii).
\end{proof}

On voit du th\'eor\`eme \ref{theorem-BP-D0(rho)} que $D_{0,\sigma}(\rho)$ est de multiplicit\'e 1 pour $\sigma\in\cD(\rho)$. Par cons\'equent, si $\tau$ est un sous-quotient irr\'eductible de $D_{0,\sigma}(\rho)$, il y a une unique sous-repr\'esentation admettant $\tau$ comme cosocle. On note cette repr\'esentation $I(\sigma,\tau)$ (au lieu de $U(\tau)$)pour accentuer que son socle est $\sigma$.

\begin{defn}
Si $S$ est une $K$-représentation de multiplicité 1 et si $\tau$
est un sous-quotient irréductible de $S$, on dit que $\tau^{I_1}$
\emph{a un relèvement} dans $S^{I_1}$, ou que $\tau^{I_1}$ \emph{se relève} dans
$S^{I_1}$, si la surjection $U(\tau)\twoheadrightarrow \tau$
induit une surjection $U(\tau)^{I_1}\twoheadrightarrow\tau^{I_1}$
où\ $U(\tau)\subseteq S$ est l'unique sous-représentation admettant
$\tau$ comme cosocle.
\end{defn}
\begin{lemma}\label{lemma-releve}
Conservons les notations du théorème \ref{theorem-BP-D0(rho)} (ii). Un sous-quotient irréductible $\tau$ de $D_{0,\sigma}(\rho)$ est tel que $\tau^{I_1}$ se relève dans $D_{0,\sigma}(\rho)^{I_1}$ si et seulement si
\[\mu_i(y_i)\in\{p-2-y_i,p-1-y_i,y_i,y_i+1\}\]
pour tout $0\leq i\leq f-1$.
\end{lemma}
\begin{proof}
Voir  \cite[corollaire 14.10]{BP}.
\end{proof}

Soient $\tau$ un poids apparaissant dans $D_{0,\sigma}(\rho)$ tel
que $\tau^{I_1}$  se relève dans $D_{0,\sigma}(\rho)^{I_1}$,
et $\mu\in\cIy$ comme dans le théorème \ref{theorem-BP-D0(rho)}
(i). Si $\rho$ est réductible, on définit
\[\begin{array}{l}\cS^-_{\lambda,\tau}=\cS_{\tau}^-:=\{i\in\cS|\
\mu_{i-1}(\lambda_{i-1}(x_{i-1}))\in\{x_{i-1},x_{i-1}+1,p-1-x_{i-1}\}\}\\
\\
\cS^+_{\lambda,\tau}=\cS_{\tau}^+:=\{i\notin\cS|\
\mu_{i-1}(\lambda_{i-1}(x_{i-1}))\in\{p-3-x_{i-1},p-2-x_{i-1},x_{i-1}\}\}.\end{array}\]
Si $\rho$ est irréductible, on définit $\cS^+_{\lambda,\tau}$ et $\cS^-_{\lambda,\tau}$ comme
ci-dessus sauf que $1\in \cS^-_{\lambda,\tau}$ (resp. $\cS^+_{\lambda,\tau}$) si et seulement
si $1\in\cS_{\lambda}$ et $\mu_{0}(\lambda_0(x_0))\in\{x_0-1,x_0,p-x_0\}$
(resp. $1\notin\cS_{\lambda}$ et
$\mu_0(\lambda_0(x_0))\in\{p-2-x_0,p-1-x_{0},x_0+1\}$).\vv

Si $\cS$ est un sous-ensemble de $\{0,\cdots,f-1\}$, on définit
$\delta_{\red}(\cS)$ (resp. $\delta_{\irr}(\cS)$) comme suit:
$i\in\delta_{\red}(\cS)$ si et seulement si $i+1\in\cS$ (resp. si
$i>0$, $i\in\delta_{\irr}(\cS)$ si et seulement si $i+1\in\cS$ et
$0\in\delta_{\irr}(\cS)$ si et seulement si $1\notin\cS$).

\begin{lemma}\label{lemma-BP-delta}
(i) Il existe un unique poids $\delta(\tau)\in\cD(\rho)$ tel que
$\tau^{[s]}\in D_{0,\delta(\tau)}(\rho)$.

(ii) Le poids $\delta(\tau)$ correspond à
$\delta_{\red}((\cS_{\lambda}\backslash \cS^-_{\lambda,\tau})\cup\cS^+_{\lambda,\tau})$ (resp. à
$\delta_{\irr}((\cS_{\lambda}\backslash \cS^-_{\lambda,\tau})\cup\cS^+_{\lambda,\tau})$) si $\rho$ est
réductible (resp. irréductible).
\end{lemma}
\begin{proof}
Le (i) est une conséquence de la construction de $D_0(\rho)$, voir \cite[proposition 13.4]{BP}. Pour le (ii), voir \cite[lemme 15.2]{BP}.
\end{proof}

\subsection{Un résultat combinatoire}

On conserve les notations  précédentes. Dans cette section, on considère deux poids $\tau_1$, $\tau_2$
apparaissant dans $D_{0,\sigma}(\rho)$ tels que:\vv
\begin{itemize}
\item[--] $\tau_1^{I_1}$, $\tau_2^{I_1}$ se relèvent dans
$D_{0,\sigma}(\rho)^{I_1}$\vspace{1mm}

\item[--] $(\tau_1,\tau_2)$ est un couple de type $(+1,j)$
\end{itemize}\vspace{1mm}
et on compare le lien entre $\delta(\tau_1)$ et $\delta(\tau_2)$,
ainsi que celui entre les places de $\tau_k$ ($k=1,2$) dans
$D_{0,\delta(\tau_k)}(\rho)$. On note $\mu_k=\mu_{\tau_k}\in\cIy$
le $f$-uplet associé à $\tau_k$ par le théorème \ref{theorem-BP-D0(rho)} (ii), et
$\theta_{k}=\mu_{\tau_k^{[s]}}$ celui associé à $\tau_{k}^{[s]}$.
Notons également $\lambda_k=\lambda_{\delta(\tau_k)}$ (resp.
$\cS_{\lambda_k}$) le $f$-uplet (resp. le sous-ensemble de
$\{0,\cdots,f-1\}$) correspondant à $\delta(\tau_k)\in\cD(\rho)$.\vv

Remarquons qu'il n'y a pas de tel couple $(\tau_1,\tau_2)$ si $f=1$. Supposons donc $f\geq 2$ dans la suite.
On fait la
convention que, si $j-1<0$ (resp. $j+1>f-1$, etc.), on l'identifie
à l'entier $j-1+f$ (resp. $j+1-f$, etc.).\vv

Donnons la liste de toutes les
possibilités pour le couple $(\tau_1,\tau_2)$ considéré.
\begin{lemma}\label{compatible-lemma}
(i) On a $\mu_{1,i}(y_i)=\mu_{2,i}(y_i)$ si $i\notin\{j-1,j\}$ et deux possibilités si $i\in\{j-1,j\}$:\vspace{1mm}
\begin{itemize}
\item[\hspace{3mm}--] ou bien $\mu_{\lambda,j}(y_j)=p-3-y_j$ et
\[(\mu_{1,j-1}(y_{j-1}),\mu_{1,j}(y_j))=(y_{j-1},y_j) \]
\[(\mu_{2,j-1}(y_{j-1}),\mu_{2,j}(y_j))=(p-2-y_{j-1},y_j+1);\]

\item[\hspace{3mm}--] ou bien $\mu_{\lambda,j}(y_j)=p-1-y_j$ et
\[(\mu_{1,j-1}(y_{j-1}),\mu_{1,j}(y_j))=(y_{j-1},p-2-y_j)\]
\[(\mu_{2,j-1}(y_{j-1}),\mu_{2,j}(y_{j}))=(p-2-y_{j-1},p-1-y_j).\]
\end{itemize}\vspace{1mm}

(ii) On a
$\theta_{1,i}(\lambda_{1,i}(r_i))=\theta_{2,i}(\lambda_{2,i}(r_i))$
si $i\notin\{j-1,j\}$, et
\[\theta_{1,j}(\lambda_{1,j}(r_j))=\theta_{2,j}(\lambda_{2,j}(r_j))+1\]
\[\theta_{1,j-1}(\lambda_{1,j-1}(r_{j-1}))+\theta_{2,j-1}(\lambda_{2,j-1}(r_{j-1}))=p.\]
\end{lemma}
\begin{proof}
(i) Le premier énoncé est facile à vérifier à l'aide du théorème \ref{theorem-BP-D0(rho)} (ii). Pour le deuxième, d'après le lemme \ref{lemma-releve}, on sait que pour tout $0\leq i\leq
f-1$,
\[\mu_{1,i}(y_i),\mu_{2,i}(y_i)\in\{p-2-y_i,p-1-y_i,y_i,y_i+1\}.\]
Comme $\mu_{2,j}(y_j)=\mu_{1,j}(y_j)+1$, on en déduit que:
\[\mu_{2,j}(y_j)\in \{p-1-y_j, y_j+1\}.\]

Si $\mu_{\lambda,j}(y_j)=p-3-y_j$, alors la condition de
compabilité  (définition \ref{definition-compatible}) entraîne que $\mu_{2,j}(y_j)=y_j+1$, et donc
$\mu_{1,j}(y_j)=y_j$. Par définition de $\cIy$ et puisque
 $\mu_{1,j-1}(y_{j-1})=p-2-\mu_{2,j-1}(y_{j-1})$, on en déduit que
$\mu_{1,j-1}(y_{j-1})=y_{j-1}$ et
$\mu_{2,j-1}(y_{j-1})=p-2-y_{j-1}$.

De même, si $\mu_{\lambda,j}(y_j)=p-1-y_j$, alors
$\mu_{2,j}(y_j)=p-1-y_j$, et les autres énoncés sont immédiats.

(ii) Par définition, on a les égalités suivantes pour tout $0\leq
i\leq f-1$:
\[\theta_{1,i}(\lambda_{1,i}(r_i))+\mu_{1,i}(\lambda_{i}(r_i))=p-1\]
\[\theta_{2,i}(\lambda_{2,i}(r_i))+\mu_{2,i}(\lambda_{i}(r_i))=p-1.\]
On conclut donc en utilisant (i):
$\mu_{1,i}(y_i)=\mu_{2,i}(y_i)$ si $i\notin\{j-1,j\}$ et que
\[\mu_{1,j}(y_j)=\mu_{2,j}(y_j)-1,\ \ \mu_{1,j-1}(y_{j-1})=p-2-\mu_{2,j-1}(y_{j-1}).\]
\end{proof}


Si $\mu\in\cIy$, on d\'efinit:
\[\cS(\mu):=\{i\in\{0,\cdots,f-1\}|\ \mu_i(x_i)\in\{p-2-x_i-\pm1,x_i\pm1\}\}.\]
(\emph{Attention}: ne pas confondre avec $\cS_{\lambda}$ o\`u $\lambda\in\cIDx$ ou $\cRDx$!)
On constate que, si $\mu,\mu'\in\cIy$ sont tels que $\mu_{i}(y_i)=\mu'_{i}(y_{i})$, alors $i+1\in \cS(\mu)$ si et seulement si $i+1\in\cS(\mu')$.

\begin{prop}\label{prop-combination}
(i) Si $i\notin\{j-1,j\}$, alors $i\in\cS_{\lambda_1}$ si et
seulement si $i\in\cS_{\lambda_2}$; si $i\in\{j-1,j\}$, alors
$i\in\cS_{\lambda_1}$ si et seulement si
$i\notin\cS_{\lambda_2}$.\vspace{1mm}

(ii) $\lambda_{1,i}(x_i)=\lambda_{2,i}(x_i)$ si et seulement si
$i\notin \{j-2,j-1,j\}$;\vspace{1mm}

(iii) Si $i\neq j-1$, alors $i\in \cS(\theta_1)$ si et seulement
si $i\in\cS(\theta_2)$; si $i=j-1$, alors $i\in\cS(\theta_1)$ si
et seulement si $i\notin\cS(\theta_2)$. 
\end{prop}

\begin{proof}
On prouve la proposition pour le cas réductible, le cas
irréductible étant analogue. \vv

(i) Par le lemme \ref{compatible-lemma} (i) et la définition de
$\cS_{\lambda,\tau_k}^+$ (resp. $\cS_{\lambda,\tau_k}^-$), on voit
que\vv
\begin{itemize}
\item[]
$\cS_{\lambda,\tau_1}^+\cap(\{0,\cdots,f-1\}\backslash\{j,j+1\})=
\cS_{\lambda,\tau_2}^+\cap(\{0,\cdots,f-1\}\backslash\{j,j+1\})$\vv

\item[]
$\cS_{\lambda,\tau_1}^-\cap(\{0,\cdots,f-1\}\backslash\{j,j+1\})=
\cS_{\lambda,\tau_2}^-\cap(\{0,\cdots,f-1\}\backslash\{j,j+1\})$.\vv
\end{itemize}
Par cons\'equent, si $i\notin\{j-1,j\}$, alors $i\in\cS_{\lambda_1}$ si et
seulement si $i\in\cS_{\lambda_2}$ d'après le lemme
\ref{lemma-BP-delta} (ii).\vv

Considérons le cas $i=j$. On a deux possibilités:
\begin{itemize}
\item[(a)] Si $j+1\in\cS_{\lambda}$, i.e.,
$\lambda_{j+1}(x_{j+1})\in\{p-3-x_{j+1},x_{j+1}+1\}$,  alors par
définition
\[\lambda_j(x_j)\in\{p-2-x_j,p-3-x_j\}\]
et on a donc deux sous-cas:
\begin{itemize}
\item[--] Si $\lambda_j(x_j)=p-2-x_j$, alors
$\mu_{\lambda,j}(y_j)=p-3-y_j$, et donc, d'après le lemme
\ref{compatible-lemma} (i), on a $ \mu_{1,j}(y_j)=y_j$ et
$\mu_{2,j}(y_j)=y_j+1$. On vérifie à partir de la définition que
$j+1\notin\cS_{\lambda,\tau_1}^-$ et
$j+1\in\cS_{\lambda,\tau_2}^-$, ce qui donne, en rappelant que
$\cS_{\lambda_i}=\delta_{\red}((\cS_{\lambda}\backslash\cS_{\lambda,\tau_i}^-)\cup\cS_{\lambda,\tau_i}^+)$,
\[j\in \cS_{\lambda_1},\ \ j\notin\cS_{\lambda_2}.\]

\item[--] Si $\lambda_j(x_j)=p-3-x_j$, alors
$\mu_{\lambda,j}(y_j)=p-1-y_j$, et donc, par le lemme
\ref{compatible-lemma} (i), on a $\mu_{1,j}(y_j)=p-2-y_j$ et
$\mu_{2,j}(y_j)=p-1-y_j$. On vérifie alors à partir de la définition que
\[j\notin\cS_{\lambda_1}, \ \  j\in\cS_{\lambda_2}.\]
\end{itemize}\vspace{1mm}
\item[(b)] Si $j+1\notin\cS_{\lambda}$, alors
$\lambda_j(x_j)\in\{x_j,x_{j+1}\}$. On a \'egalement deux sous-cas à distinguer et le
même argument donne:\vspace{1mm}
\begin{itemize}
\item[--] si $\lambda_{j}(x_j)=x_j$, alors
$\mu_{\lambda,j}(y_j)=p-1-y_j$, et puis $j\in \cS_{\lambda_1}$,
$j\notin\cS_{\lambda_2}$;\vspace{0.8mm}

\item[--] si $\lambda_j(x_j)=x_j+1$, alors
$\mu_{\lambda,j}(y_j)=p-3-y_j$, et puis
$j\notin\cS_{\lambda_1}$, $j\in\cS_{\lambda_2}$.
\end{itemize}
\end{itemize}\vspace{2mm}
Ceci permet de conclure dans le cas où\ $i=j$.

Le même raisonnement (plus facile) donne, lorsque $i=j-1$:
\[j-1\in\cS_{\lambda_1}\Longleftrightarrow j-1\notin\cS_{\lambda_2}.\]

(ii) Par (i), on a:
\[\cS_{\lambda_1}\cap(\{0,\cdots,f-1\}\backslash\{j-1,j\})=\cS_{\lambda_2}
\cap(\{0,\cdots,f-1\}\backslash\{j-1,j\}),\] donc d'après le lemme \ref{combination-corollary}, on a $\lambda_{1,i}(x_i)=\lambda_{2,i}(x_i)$ si
$i\notin\{j-2,j-1,j\}$. De plus, (i) implique
$\lambda_{1,i}(x_i)\neq\lambda_{2,i}(x_i)$ si $i\in\{j-1,j\}$.
Il reste donc à vérifier que
$\lambda_{1,j-2}(x_{j-2})\neq\lambda_{2,j-2}(x_{j-2})$, ce qui est
une conséquence de la définition et de (i) pour $j-1$.\vv

(iii) La conclusion pour $i\notin\{j-2,j-1,j\}$ est triviale
puisqu'alors $\theta_{1,i}(y_i)=\theta_{2,i}(y_i)$.

On suppose que $f\geq 3$, le cas $f=2$ étant conséquence de \cite[\S16]{BP}. On a donc $j+1\notin\{j-1,j\}$, et
d'après (i), $j+1\in\cS_{\lambda_1}$ si et seulement si
$j+1\in\cS_{\lambda_2}$. On en déduit, puisque
$\lambda_{1,j}(x_j)\neq \lambda_{2,j}(x_j)$, que
\[\lambda_{1,j}(x_j)=\lambda_{2,j}(x_j)\pm1.\]
Ensuite, comme
$\theta_{1,j}(y_j),\theta_{2,j}(y_j)\in\{p-2-y_j,p-1-y_j,y_j,y_j+1\}$,
le fait que
$\theta_{1,j}(\lambda_{1,j}(r_{j}))=\theta_{2,j}(\lambda_{2,j}(r_{j}))+1$
(lemme \ref{compatible-lemma} (ii)) force que:
\[\theta_{1,j}(y_j)=\theta_{2,j}(y_j).\]
Ceci permet de conclure dans le cas $i=j$. Le cas  $i=j-2$ s'en
déduit puisque $\theta_{1,j-3}(y_{j-3})=\theta_{2,j-3}(y_{j-3})$ (m\^eme si $f=3$).

Il reste à vérifier que $j-1\in \cS(\theta_1)$ si et seulement si
$j-1\notin\cS(\theta_2)$. Ceci est une conséquence de ce que l'on a prouvé et du fait suivant
(facile à vérifier): \emph{si $\theta_1, \theta_2\in\cIy$ sont
tels que pour tout $0\leq i\leq f-1$:
\[\theta_{1,i}(y_i), \theta_{2,i}(y_i)\in\{p-2-y_i,p-1-y_i,y_i,y_i+1\},\]
alors $\theta_1=\theta_2$ si et seulement si
$\cS(\theta_1)=\cS(\theta_2)$.}
\end{proof}

Soit $\tau$ un poids apparaissant dans
$D_{0,\sigma}(\rho)$. Par réciprocité de Frobenius, on a une
surjection naturelle $\Ind_I^K\chi_{\tau}^{s}\twoheadrightarrow
I(\delta(\tau),\tau^{[s]})$ de telle sorte que $\delta(\tau)$ correspond à
un élément $\xi\in\cPy$. Rappellons que
\[J(\xi)=\{i\in\{0,\cdots,f-1\}|\ \xi_i(y_i)\in\{p-2-y_i,p-1-y_i\}\}.\]
On vérifie facilement que \[i\in J(\xi) \Longleftrightarrow
\theta_i(y_i)\in\{y_i,y_i+1\} \Longleftrightarrow i+1\notin
\cS(\theta).\]

\begin{cor}\label{corollaire-combination-J}
Avec les notations de la proposition \ref{prop-combination} et $\xi_i$ \'etant l'élément correspondant à $\delta(\tau_i)$ comme ci-dessus ($i=1,2$), on a:
\[J(\xi_1)\cap (\{0,\cdots,f-1\}\backslash\{j-2\})=J(\xi_2)\cap (\{0,\cdots,f-1\}\backslash\{j-2\})\]
et que $j-2\in J(\xi_1)$ si et seulement si $j-2\notin J(\xi_2)$.
\end{cor}
\begin{proof}
C'est une simple traduction de la
proposition \ref{prop-combination} (iii).
\end{proof}

\section{Constructions de représentations supersingulières}\label{section-exemple}

 Si $\rho$ est une représentation continue
générique de $\Gal(\bQp/F)$ de dimension 2 sur $\bFp$ telle que $p\in F^{\times}$ agisse trivialement sur $\det(\rho)$, on
lui a associé une famille de diagrammes $D(\rho,r)$ (cf.
n$^{\circ}$\ref{subsection-PoidsDiamond}). Par \cite[théorème 9.8]{BP}, on peut aussi lui associer  un ensemble non vide de
représentations lisses admissibles de $G$. On note $S(\rho,r)$ l'ensembe des repr\'esentations obtenues de cette fa\c{c}on (associ\'ees \`a $D(\rho,r)$). Lorsque
$f=1$, il est connu (\cite[\S20]{BP}) que $S(\rho,r)$ est r\'eduit \`a un singleton, c'est-\`a-dire, \`a isomorphisme pr\`es le diagramme $D(\rho,r)$ d\'etermine une unique repr\'esentation de $G$ par la construction.
L'objet de ce paragraphe est de montrer que ce n'est plus vrai lorsque $f\geq 2$: en faisant un choix pour $r$ ne suffit pas \`a d\'eterminer une unique repr\'esentation lisse admissible de $G$ associ\'ee \`a $\rho$.\vv

\subsection{Préliminaires}\label{subsection-prelimilaire}
Dans cette section, on donne une construction générale de représentations
lisses admissibles de $G$ (\`a partir d'une telle repr\'esentation fix\'ee) .
Elle sera utilisée dans les sections suivantes.\vv

On fixe $\pi$ une représentation lisse admissible de $G$ telle que
$p\in F^{\times}$ agisse trivialement. Soit $\Omega$ une repr\'esentation lisse admissible de $G$ telle que $\pi\hookrightarrow\Omega$ et telle que $\Omega|_K$ soit isomorphe \`a $\rInj_K\rsoc_K(\pi)$, une enveloppe injective de $\rsoc_K(\pi)$ dans la cat\'egorie $\Rep_K$.


Soient $M$ une sous-$I$-repr\'esentation de $\pi$ de dimension finie et $v\in M$ un
vecteur propre de $\cH$ tel que $v$ n'appartienne pas \`a $\rad_I(M)$. Une telle vecteur toujours existe puisque $\rad_I(M)\subsetneq M$. Soit $\chi$ le caract\`ere donnant l'action de $I$ sur $\bFp v$. Notons $\overline{v}$ l'image de $v$ dans $M/\rad_I(M)$ et
\[\beta: M\twoheadrightarrow M/\rad_I(M)\twoheadrightarrow \bFp\overline{v}\]
la projection naturelle avec $\ker\beta$ son noyau. Supposons que $(M,v)$ satisfait aux conditions suivantes:\vspace{1mm}

(S1)  $v\notin\Sigma(M)$,
o\`u $\Sigma(M)$ est la sous-$I$-représentation de $\Omega$ engendrée
par $\ker\beta$, $\Pi(M)$, $\Omega^{I_1}$.\vspace{1mm}

(S2) Il existe $h\in \Omega^{I_1}$ un vecteur non nul propre de $\cH$ de caract\`ere $\chi^s$.\vv

Alors, \`a partir des donn\'ees $(\pi,\Omega, M,v,h)$, on va constuire une famille de repr\'esentations lisses admissibles de $G$ comme suit. \vv

\textbf{Étape 1}. 
Soit $V_{\chi}$ une sous-$I$-repr\'esentation de $\Omega$ qui est isomorphe à
 $\rInj_I\chi$ et telle que:
\[V_{\chi}^{I_1}=\bFp\Pi(h).\]
($V_{\chi}$ existe parce que $\bFp \Pi(h)\hookrightarrow\Omega$ et $\Omega|_I$ est un object injectif dans $\Rep_I$.) Alors l'injection $\bFp v\hookrightarrow \Omega/\Sigma(M)$ induit,
par injectivité de $V_{\chi}$, un morphisme $\Omega/\Sigma(M)\ra
V_{\chi}$
 qui envoie l'image
$\overline{v}$ de $v$ vers $\Pi(h)$. En le composant avec les morphismes
naturels, on obtient un endomorphisme
$I$-équivariant de $\Omega$ que l'on note $\phi$:
\[\phi:\Omega\twoheadrightarrow \Omega/\Sigma(M)\ra
V_{\chi}\hookrightarrow \Omega,\]  On note $\Phi$
l'ensemble des $\phi$ ainsi construits.

\begin{rem}
(i) Il existe en g\'en\'eral beaucoup de tels endomorphismes
$\phi$.

(ii) Puisque $\Omega^{I_1}\subset\ker\phi$, pour tout $v\in\Omega$, il existe $n\gg0$ d\'ependant de $v$ tel que $\phi^n(v)=0$.
Par conséquent, pour tout $a\in\bFp$, $1+a\phi$ est un
automorphisme de $\Omega$ dont l'inverse est
$\sum_{n=0}^{\infty}(-a\phi)^n$.
\end{rem}

\textbf{Étape 2}: On définit, pour tout $a\in\bFp$, une
action de la matrice $\smatr01p0$ sur $\Omega$:
\begin{equation}\label{equation--exem-define-Pi}
\begin{array}{rllc}
\Pi_{\phi,a}:& \Omega&\lra & \Omega\\
&x&\longmapsto &(1+a\phi)^{-1}\cdot\Pi\cdot (1+a\phi)(x).
\end{array}
\end{equation}
On vérifie facilement que $\Pi_{\phi,a}$ est $I$-équivariant et que
$(\Pi_{\phi,a})^2=\id_{\Omega}$.

Comme $\bFp\Pi(v)+\Omega^{I_1}\subset
\Sigma(M)\subset\ker\phi$, on a par définition:
\[\begin{array}{rll}\Pi_{\phi,a}(v)&=&(1+a\phi)^{-1}\Pi(v+a\Pi(h))\\
&=&\sum\limits_{n=0}^{\infty}(-a\phi)^{n}(\Pi(v)+ah)\\
&=&\Pi(v)+ah.
\end{array}\]
Ce calcul montre aussi que, si $x\in \Omega$ vérifie à la fois que $x\in\ker(a\phi)$ et $\Pi(x)\in\ker(a\phi)$,
alors $\Pi_{\phi,a}(x)=\Pi(x)$. En particulier,
$\Pi_{\phi,0}=\Pi$.\vv


\textbf{Étape 3}: Puisque $\Pi_{\phi,a}$ définit une action de la
matrice $\smatr{0}1p0$ sur $\Omega$ compatible avec l'action de $I$, on obtient d'après \cite[\S9]{BP}
une représentation lisse admissible de $G$ que l'on note
$(\Omega, \Pi_{\phi,a})$, ou simplement $\Omega_{\phi,a}$,  telle que
\[\Omega_{\phi,a}|_{KZ}=\Omega|_{KZ}=\rInj_K\rsoc_K(\pi).\] On définit
$\pi_{\phi,a}$ comme la sous-$G$-représentation de $\Omega_{\phi,a}$
engendrée par $M+\pi^{I_1}$.\vv

\'Evidemment, les repr\'esentations $\{\pi_{\phi,a},\ \phi\in\Phi, a\in\bFp\}$ de $G$ associ\'ees \`a $\pi$ comme ci-dessus sont lisses admissibles, admettant un caract\`ere central. Bien s\^ur, ces repr\'esentations $\pi_{\phi,a}$ ne sont pas forc\'ement non isomorphes. Quand m\^eme, on verra dans \S\ref{subsection-cas-f=3} et \S\ref{subsection-cas-f=2} que sous des conditions suppl\'ementaires sur $(\pi,M,v,h)$, on aura $\pi_{\phi,a}\ncong \pi$ avec $a$ bien choisi.\vv

Donnons un crit\`ere pour que $(M,v)$ v\'erifie les conditions (S1) et (S2). Pour $S$ une $I$-repr\'esentation de dimension finie, notons $r^+(S)$ (resp. $r^-(S)$) la longueur de Loewy de $S$ en tant que $U^+$-repr\'esentation (resp. $U^-$-repr\'esentation).

\begin{lemma}\label{lemma=S1S2}
(i) Soit $\langle I\cdot v\rangle\subset M$ la sous-repr\'esentation engendr\'ee par $v$. Si
\begin{equation}\label{equation-r+r-}
r^+(\langle I\cdot v\rangle)>\max\{r^-(M), r^+(\ker(\beta))\},\end{equation} alors $(M,v)$ satisfait \`a (S1).

(ii) Soit $\sigma$ un poids apparaissant dans $\rsoc_K(\pi)$. Alors $\chi^s$ appara\^it dans $(\rInj_K\sigma)^{I_1}$ si et seulement si $\sigma$ appara\^it dans la repr\'esentation $\Ind_I^K\chi^s$. S'il en est ainsi, alors la multiplicit\'e de $\chi^s$ dans $(\rInj_K\sigma)^{I_1}$ est \'egale \`a 1.
\end{lemma}
\begin{proof}
(i) Par hypoth\`ese on a en particulier $r^+(\langle I\cdot v\rangle)>1$ puisque $r^-(M)\geq 1$. D'autre part, il r\'esulte de la d\'efinition que $r^-(M)=r^+(\Pi(M))$, que $r^+(S+S')=\max\{r^+(S),r^+(S')\}$ pour deux $I$-repr\'esentations $S$ et $S'$ de dimension finie et que $r^+(S)\leq r^+(S')$ si $S\subset S'$. Le r\'esultat s'en d\'eduit en utilisant la condition (\ref{equation-r+r-}) car par d\'efinition $\Sigma(M)=\Pi(M)+\ker(\beta)+\Omega^{I_1}$ et que $r^+(\Omega^{I_1})=1$.

(ii) C'est une cons\'equence de la r\'eciprocit\'e de Frobenius: $\chi^s$ s'injecte dans $\rInj_K\sigma$ si et seulement s'il existe un morphisme $K$-\'equivariant non trivial $\Ind_I^K\chi^s\ra \rInj_K\sigma$. Le dernier \'enonc\'e d\'ecoule du fait que $(\rInj_K\sigma)^{I_1}$ est de multiplicit\'e 1 (cf. \cite[\S4]{BP}).
\end{proof}


\subsection{Les cas $f\geq 3$}\label{subsection-cas-f=3}
Soient $\rho:\Gal(\bQp/F)\ra \GL_2(\bFp)$ une représentation
continue générique (cf. définition \ref{definition-generique}) telle que $p\in F^{\times}$ agit trivialement sur $\det(\rho)$ et $D(\rho,r)$ un diagramme de Diamond associ\'e.  Dans cette section, on montre que l'ensemble $S(\rho,r)$ de repr\'esentations lisses admissibles de $G$ associ\'e \`a $D(\rho,r)$ n'est pas r\'eduit \`a un singleton lorsque $f\geq 3$. \vv

Supposons donc $f\geq 3$.  Supposons de plus que:
\begin{itemize}
\item[(I)] ou bien $f=2m+1$ avec $m\geq 1$ et $\rho$ est
irréductible,

\item[(II)] ou bien $f=2m+2$ avec $m\geq 1$ et $\rho$ est
réductible scindée.
\end{itemize}
On fixe $\pi\in S(\rho,r)$ une repr\'esentation de $G$ associ\'ee \`a $D(\rho,r)$.

\begin{lemma}\label{lemma-exemple-poids}
Sous les hypothèses précédentes, il existe un poids
$\sigma\in\cD(\rho)$ tel que $\sigma^{[s]}\in\cD(\rho)$ et tel que
\[\mu_{\lambda}=(p-3-y_0,\cdots,p-3-y_{f-1}),\]
où\ $\lambda\in\cIDx$ ou $\cRDx$  est le
$f$-uplet correspondant à $\sigma$ et où $\mu_{\lambda}\in\cIy$ est
le $f$-uplet associé à $\lambda$ défini au
\S\ref{subsection-PoidsDiamond}.
\end{lemma}
\begin{proof}
Il suffit de prendre $\sigma$ comme le poids de Diamond
correspondant à
\begin{itemize}
\item[--] dans le cas (I),
\[\lambda=(p-1-x_0,x_1+1,p-2-x_2,\cdots,x_{2m-1}+1,p-2-x_{2m});\]

\item[--] dans le cas (II),
\[\lambda=(x_0+1,p-2-x_1,\cdots,x_{2m}+1,p-2-x_{2m+1}).\]
\end{itemize}
L'énoncé concernant $\mu_{\lambda}$ est alors direct par définition (cf. \S\ref{subsection-PoidsDiamond}).
\end{proof}

On fixe un poids $\sigma$ comme dans le lemme \ref{lemma-exemple-poids}. D'après le théorème \ref{theorem-BP-D0(rho)} (ii), pour tout $1\leq j\leq f-1$, il
existe un poids $\tau_j$ apparaissant dans $D_{0,\sigma}(\rho)$
tel que $(\sigma,\tau_j)$ soit un couple de type $(+1,j)$. De manière explicite, $\tau_j$ est le sous-quotient irréductible de $D_{0,\sigma}(\rho)$ correspondant à
\[(\cdots,y_{j-2},p-2-y_{j-1},y_j+1,y_{j+1},\cdots)\in\cIy.\] Pour
simplifier les notations, on fixe $j\in\{0,\cdots,f-1\}$ et on
pose $\tau=\tau_j$.

Rappelons que $\delta(\tau)$ désigne l'unique poids de Diamond tel
que $\tau^{[s]}$ soit un sous-quotient irréductible de
$D_{0,\delta({\tau})}(\rho)$. Puisque $\delta(\tau)$ est un
sous-quotient irréductible de $\Ind_I^K\chi_{\tau}^s$, on peut lui associer un
$f$-uplet $\xi\in \cPx$ (cf. n$^{\circ}$\ref{subsection-principal}) tel que, si l'on écrit
$\tau=(s_0,\cdots,s_{f-1})\otimes\eta$, alors
\[\delta(\tau)=(\xi_0(s_0),\cdots,\xi_{f-1}(s_{f-1}))\otimes{\det}^{e(\xi)(s_0,\cdots,s_{f-1})}\eta.\]

\begin{cor}\label{corollaire--exemple-J(tau)}
Dans le cas (I) ou (II), on a
$J(\xi)=\{0,\cdots,f-1\}\backslash\{j-2\}$ (où\ l'on identifie $j-2$
avec $j-2+f$ si $j-2<0$).
\end{cor}
\begin{proof}
C'est une conséquence du corollaire
\ref{corollaire-combination-J} en remarquant que
$J(\delta(\sigma))=J(\sigma^{[s]})=\{0,\cdots,f-1\}$ et que $(\sigma,\tau)$ est un couple de type $(+1,j)$.
\end{proof}

Rappelons que (\S\ref{subsection-PoidsDiamond}) $I(\sigma,\tau)$ d\'esigne l'unique sous-repr\'esentation de $D_{0,\sigma}(\rho)$ qui admet $\tau$ comme cosocle. Dans notre cas, $I(\sigma,\tau)$ est isomorphe \`a l'unique extension non triviale de $\tau$ par $\sigma$. D'après le lemme \ref{lemma-E(deux-char)}, la $K$-représentation
$I(\sigma,\tau)\subset D_{0,\sigma}(\rho)$ contient une unique
sous-$I$-représentation $M_{\tau}$ isomorphe à
$E_{j-1}(\chi_{\sigma},\chi_{\tau},r_{j-1}'+1)$ o\`u l'on a \'ecrit \[\sigma=(r_0',\cdots,r_{j}',\cdots,r'_{f-1})\] à torsion près. On pose $\chi':=\chi_{\tau}\alpha^{-p^{j}}$ et
\[W:=\Ind_I^K\Pi(M_{\tau})\cong\Ind_I^K\Pi(E_{j-1}(\chi_{\sigma},\chi_{\tau},r'_{j-1}+1)).\]

\begin{cor}
Soit $\omega$ un sous-quotient irréductible de
$\Ind_I^K\Pi(\chi')$, alors la $K$-représentation $W_{\omega}$
(cf. lemme \ref{lemma-W-vecteurs-k}) admet
$\sigma^{[s]}=\delta(\sigma)$ et $\delta(\tau)$ comme sous-représentations si et seulement si
\[\{0,\cdots,f-1\}\backslash\{j-1,j-2\}\subseteq J(\omega).\]
\end{cor}
\begin{proof}
C'est une conséquence des corollaires \ref{corollary-W-U(tau)} et
\ref{corollaire--exemple-J(tau)}.
\end{proof}

\begin{lemma}\label{lemma-exemple-f>2-F_0v}
(i) Il existe un poids de Diamond $\sigma'$ tel que $\chi'=\chi_{\sigma'}$. Si l'on note $\lambda'$ le $f$-uplet correspondant à $\sigma'$, alors
$|\ell(\lambda')-\ell(\lambda)|=1$.

(ii) Soit $v_{\tau}\in M_{\tau}$ un vecteur propre non nul de $\cH$ de
caractère $\chi'$. Posons
\[F_{0}:=\summ_{\lambda\in\F_q}\matr{[\lambda]}110\Pi(v_{\tau})\in \pi.\]
Alors ou bien $F_{0}=0$, ou bien $\langle K\cdot
F_{0}\rangle$ est irréductible isomorphe à $\sigma'$.
\end{lemma}
\begin{proof}
(i) Remarquons que  si l'on écrit
$\chi'=(s_0,\cdots,s_{f-1})\otimes\eta'$, alors $0\leq s_{j}\leq p-3$. Soit $\sigma'$ l'unique
poids de dimension $\leq q-2$ tel que $\chi_{\sigma'}=\chi'$.
On vérifie que $(\sigma,\sigma')$ est un couple de type
$(-1,j)$ et que $\sigma'$ n'apparaît pas dans $D_{0,\sigma}(\rho)$ (cf. théorème \ref{theorem-BP-D0(rho)}).
Ceci montre que $\sigma'\in \cD(\rho)$ par maximalité de
$D_0(\rho)$ (\cite[proposition 13.1]{BP}). Le deuxième énoncé
est immédiat.\vv

(ii) Supposons $F_{0}\neq 0$. Il suffit de
démontrer que l'image du morphisme naturel
\[W_{\sigma'}\hookrightarrow \Ind_I^K\Pi(M_{\tau})\twoheadrightarrow \langle K\cdot M_{\tau}\rangle\hookrightarrow \pi\]
est isomorphe à
$\sigma'$, où $W_{\sigma'}\subset W$ est défini dans le lemme
\ref{lemma-W-vecteurs-k} (avec $\sigma'$ \'etant vu comme sous-quotient de $\Ind_{I}^K\Pi(\chi')$). Si l'on note $W_1$ le noyau du
morphisme
\[\Ind_I^K(\chi_{\sigma}^s\oplus\chi_{\tau}^s)\ra \pi,\]
alors le morphisme $W_{\sigma'}\ra\pi$ se factorise à travers
$W_{\sigma'}/(W_{\sigma'}\cap W_1)$, et puis on vérifie \`a l'aide de l'exemple \ref{exemple-1} que $W_{\sigma'}/ (W_{\sigma'} \cap W_1)$ n'admet pas
de poids de Diamond autre que $\sigma'$ comme sous-quotient et que $\sigma'$ apparaît dans $W_{\sigma'}$ avec multiplicité 1. Ceci
permet de conclure.
\end{proof}

Dans le lemme \ref{lemma-exemple-f>2-F_0v}, il est clair que le poids $\sigma'$ est uniquement déterminé par $\sigma$ et $j$. Choisissons une repr\'esentation lisse admissible $\Omega$ de $G$ telle que $\pi\hookrightarrow \Omega$ et telle que $\Omega|_K\cong\rInj_K\rsoc_K(\pi)=\oplus_{\sigma\in\cD(\rho)}\rInj_K\sigma$. Le lemme \ref{lemma=S1S2} donne pour $(M_{\tau},v_{\tau})$
\begin{enumerate}
\item[(S1)] $v_{\tau}\notin \Sigma(M_{\tau})$, car $r^-(M_v)=1$ et $M_v=\langle I\cdot v_{\tau}\rangle$ de telle sorte que $r^+(\langle I\cdot v_{\tau}\rangle)> r^+(\ker\beta)$;\vspace{1mm}

\item[(S2)] il existe
$f_{\sigma'}\in (\rInj_K\sigma')^{I_1}\subset\Omega^{I_1}$  un vecteur (unique à scalaire près) non nul  propre de $\cH$ de caractère
$\chi'^s=\chi_{\sigma'}^s$.
\end{enumerate}\vspace{1mm}
On obtient alors par la construction dans \S\ref{subsection-prelimilaire} une famille de repr\'esentations lisse admissibles de $G$: $\{\pi_{\phi,a},\ {\phi\in\Phi, a\in\bFp}\}$.

\begin{lemma}\label{lemma-exemple-f=3-in-S}
Pour tout $\phi\in\Phi$ et $a\in\bFp$, on a $\pi_{\phi,a}\in S(\rho,r)$.
\end{lemma}
\begin{proof}
Reprenons les notations de \S\ref{subsection-prelimilaire} concernant la construction de $\pi_{\phi,a}$.
Par construction, $\pi_{\phi,a}$ est la sous-$G$-repr\'esentation de $\Omega_{\phi,a}$ engendr\'ee par $M_{\tau}+\pi^{I_1}$. En particulier, on a $D_1(\rho)\hookrightarrow \pi_{\phi,a}$. D'apr\`es \cite[lemme 19.7]{BP} (qui ne d\'epend de la structure de $D_0(\rho)$), on trouve que $D_0(\rho)\hookrightarrow\pi_{\phi,a}$ et que $\pi_{\phi,a}$ est engendr\'ee par $D_1(\rho)$. D'ailleurs, il r\'esulte de la construction que
$\Pi_{\phi,a}(x)=\Pi(x)$ si $x\in D_1(\rho)$, d'o\`u $D(\rho,r)$ s'injecte dans $(\pi_{\phi,a}|_{KZ},\pi_{\phi,a}|_N,\mathrm{can})$ en tant que diagrammes. Cela permet de conclure.
\end{proof}

Le r\'esultat suivant répond négativement à la question (Q2) (de l'introduction).

\begin{theorem}\label{theorem-exem-f>3-irr}
Dans le cas (I), il existe (au moins) deux \'el\'ements de $S(\rho,r)$ qui sont non isomorphes.
\end{theorem}
\begin{proof}
Reprenons les notations du \S\ref{subsection-prelimilaire} concernant la construction de $\pi_{\phi,a}$ et choisissons une constante $a\in\bFp$ telle que
\begin{equation}\label{equation-neq0}F_{0}+a\summ_{\lambda\in\F_q}\matr{[\lambda]}110f_{\sigma'}\left\{ {\begin{array}{ll}
\neq0 & \mathrm{si}\ F_{0}=0  \\
=0& \mathrm{si}\ F_{0}\neq 0.
\end{array}}\right.\end{equation}
($a$ existe car
$\sum_{\lambda\in\F_q}\smatr{[\lambda]}110f_{\sigma'}$ n'est pas
nul dans $\pi$).
Il suffit de montrer que $\pi\cong\pi_{\phi,a}$ en tant que $G$-repr\'esentations (avec $\phi\in\Phi$ quelconque). Par l'absurde, soit $\psi:\pi\simto \pi_{\phi,a}$ un isomorphisme $G$-\'equivariante. Comme $M_{\tau}$ est contenu dans la $K$-repr\'esentation $I(\sigma,\tau)$ et comme \[\dim_{\bFp}\Hom_{\bFp}(I(\sigma,\tau),\Omega)=1,\]
on peut
supposer $\psi_a|_{M_{\tau}}=\id$. D'ailleurs, puisque $\phi(x)=0$ pour tout $x\in\ker(\beta)+\Pi(\ker(\beta))$, on a $\Pi(x)=\Pi_{\phi,a}(x)$ (dans $\Omega$). Or, par construction (\ref{equation-neq0}), les deux $K$-repr\'esentations \[\langle K\cdot\Pi(M_{\tau})\rangle/\langle K\cdot\Pi(\ker(\beta))\rangle\]
et
\[\langle K\cdot\Pi_{\phi,a}(M_{\tau})\rangle/\langle K\cdot\Pi(\ker(\beta))\rangle\]
n'admettent pas \`a la fois $\sigma'$ comme sous-quotient. Cela donne une contradiction parce que
\[\psi(\langle K\cdot \Pi(M_{\tau})\rangle)=\langle K\cdot\Pi_{\phi,a}(M_{\tau})\rangle\]
par la $G$-\'equivariance de $\psi$.
\end{proof}
\vv

Supposons maintenant que l'on est dans le cas (II). D'après
\cite[théorème \ref{theorem-BP-D0(rho)} (iii)]{BP}, on a une décomposition de diagrammes
\[D(\rho,r)=\oplus_{\ell=0}^f(D_{0,\ell}(\rho),D_{1,\ell}(\rho),r_{\ell})\]
où $D_{i,\ell}(\rho):=\oplus_{\ell(\sigma)=\ell}D_{i,\sigma}(\rho)$ pour $i\in\{0,1\}$.
Par \cite[théorème 19.9]{BP}, on peut supposer que $\pi$ est une somme directe de sous-repr\'esentations $\pi_{\ell}$  pour $\ell\in\{0,\cdots,f\}$ v\'erifiant\vv

\begin{itemize}
\item[(a)]
$\rsoc_{K}\pi_{\ell}=\bigoplus_{\substack{\sigma''\in\cD(\rho)\\
\ell(\sigma'')=\ell}}\sigma''$ où $\ell(\sigma'')$ est d\'efini par (\ref{equation-ell});
\vspace{1mm}

\item[(b)]
$(D_{0,\ell}(\rho),D_{1,\ell}(\rho),r_{\ell})\hookrightarrow
(\pi_{\ell}^{K_1},\pi_{\ell}^{I_1},\mathrm{can})$;\vspace{1mm}

\item[(c)] $\pi_{\ell}$ est engendrée par $D_{1,\ell}(\rho)$ en tant que $G$-représentation.\vv
\end{itemize}

\begin{rem}\label{remark-F0=0}
Sous ces conditions pr\'ec\'edentes sur $\pi$, on a $F_0=0$ dans le lemme \ref{lemma-exemple-f>2-F_0v} car $F_0\in \pi_{\ell(\sigma)}$ et $\ell(\sigma')\neq \ell(\sigma)$.
\end{rem}

Le théorème suivant va répondre \emph{négativement} à la question (Q3).
\begin{theorem}\label{theorem-exem-f>3-red}
Il existe une représentation lisse admissible $\pi'$ de $G$
vérifiant les propriétés suivantes:\vspace{2mm}

\begin{itemize}
\item[(a')]
$\rsoc_{K}\pi'=\bigoplus_{\sigma''\in\cD(\rho)}\sigma''$;\vspace{1mm}

\item[(b')] $(D_{0}(\rho),D_{1}(\rho),r)\hookrightarrow
(\pi'^{K_1},\pi'^{I_1},\mathrm{can})$;\vspace{1mm}

\item[(c')] $\pi'$ est engendrée par $D_{1}(\rho)$ en tant que $G$-représentation;\vspace{1mm}

\item[(d')] $\pi'$ n'est pas semi-simple, en particulier,
$\pi'$ n'est pas une somme directe
$\pi'=\oplus_{\ell=0}^f\pi_{\ell}'$ avec $\pi_{\ell}'$ des
sous-représentations vérifiant (a)-(c) ci-dessus.\vspace{1mm}
\end{itemize}
\end{theorem}
\begin{proof}
On va construire $\pi'$ \`a partir de $\pi$ en modifiant la construction dans n$^{\circ}$\ref{subsection-prelimilaire}. 
Soit $\Omega|_{K}=\oplus_{\ell=0}^f\Omega_{\ell}$ une d\'ecomposition de $\Omega$ (en tant que $K$-repr\'esentation) telle que $\Omega_{\ell}$ est isomorphe \`a $\rInj_{K}\rsoc_K(\pi_{\ell})$
et l'injection $\pi\hookrightarrow \Omega$ que l'on a fix\'e induit $\pi_{\ell}\hookrightarrow \Omega_{\ell}$ pour tout $\ell\in\{0,\cdots,f\}$.
Comme $M_{\tau}, \Pi(M_{\tau})\subset\pi_{\ell(\sigma)}\subset\Omega_{\ell(\sigma)}$, il existe une
sous-$I$-représentation $E$ de $\Omega$ contenant
\[\bigoplus_{\ell\neq\ell(\sigma)}\Omega_{\ell}
+\Sigma(M_{\tau}),\]
mais pas $v_{\tau}$. Soit $V_{\chi'}$ une
sous-$I$-représentation de $\rInj_K\sigma'\subset\Omega_{\ell(\sigma')}$ isomorphe à
$\rInj_I\chi'$ telle que $V_{\chi'}^{I_1}=\bFp \Pi(f_{\sigma'})$.
En choisissant\vv
\begin{itemize}
\item[--] d'une part, un endomorphisme $I$-\'equivariant $\phi:\Omega\ra\Omega$  qui
se factorise par
\[\Omega\ra \Omega/E\ra V_{\chi'}\ra \Omega,\]

\item[--] d'autre part, une constante $a\in\bFp$ non nulle
\end{itemize}
on obtient donc une action de $\smatr 01p0$ sur $\Omega$, i.e. on
définit $\Pi_{\phi,a}:\Omega\ra\Omega$ comme dans
(\ref{equation--exem-define-Pi}). On en déduit alors une
représentation lisse admissible $\Omega_{\phi,a}$ de $G$. Enfin,
on définit $\pi_{\phi,a}$ comme la sous-$G$-représentation de
$\Omega_{\phi,a}$ engendrée par $D_1(\rho)$.\vspace{1mm}

Montrons que la représentation
$\pi':=\pi_{\phi,a}$ satisfait aux conditions demandées. Par
construction, $\pi'$ satisfait (a'), (b'), (c'). Pour vérifier la condition (d'), il suffit de montrer que $\pi'$ ne poss\`ede pas de sous-$G$-repr\'esentation qui a pour $K$-socle $\oplus_{\ell(\sigma'')=\ell(\sigma)}\sigma''$. Supposons par l'absurde $\pi'_{\ell(\sigma)}$ une telle sous-repr\'esentation.
Comme l'on a vu dans la remarque \ref{remark-F0=0},
\[F_0=\summ_{\lambda\in\F_q}\matr{[\lambda]}110\Pi(v_{\tau})=0\]
et comme et $\Pi_{\phi,a}(v_{\tau})=\Pi(v_{\tau})+af_{\sigma'}$ et $a\neq 0$, on voit que
\[F_0':=\summ_{\lambda\in\F_q}\matr{[\lambda]}110\Pi_{\phi,a}(v_{\tau})\neq 0.\]
et que $F_0'$ engendre $\sigma'$ sous l'action de $K$.
Autrement dit, $\pi_{\ell(\sigma)}'$ admet $\sigma'$ dans son socle, ce qui donne une contradiction et permet de conclure.
\end{proof}

\subsection{Le cas $f=2$ et $\rho$ irréductible}\label{subsection-cas-f=2}
Supposons maintenant que
$f=2$ et fixons $\rho$ une représentation continue (irréductible)
$\Gal(\bQp/F)\ra\GL_2(\bFp)$ telle que
\[\rho|_{I(\bQp/F)}\cong\matr{\omega_4^{r_0+1+p(r_1+1)}}00{\omega_4^{p^2(r_0+1)+p^3(r_1+1)}}\]
avec $1\leq r_0\leq p-2$, $0\leq r_1\leq p-3$.
Alors l'ensemble des poids de Diamond est donné par:
\[\begin{array}{rll}\sigma_1&:=&(r_0,r_1)\vspace{2mm}\\
\sigma_2&:=&(r_0-1,p-2-r_1)\otimes{\det}^{p(r_1+1)}\vspace{2mm}\\
\sigma_3&:=&(p-1-r_0,p-3-r_1)\otimes{\det}^{r_0+p(r_1+1)}\vspace{2mm}\\
\sigma_4&:=&(p-2-r_0,r_1+1)\otimes{\det}^{r_0+p(p-1)},\vspace{2mm}\end{array}\]
et la représentation $D_0(\rho)$ est (à torsion près)
\[\begin{array}{rccccll}D_{0,\sigma_1}(\rho)&:=&\sigma_1&\textbf{---}&S_1&\textbf{---}&(p-3-r_0,p-1-r_1)\\
&&&&\oplus&\\
D_{0,\sigma_2}(\rho)&:=&\sigma_2 &\textbf{---}&S_2&\textbf{---}&(p-r_0,r_1-1)\\
&&&&\oplus&\\
D_{0,\sigma_3}(\rho)&:=&\sigma_3&\textbf{---}&S_3&\textbf{---}&(r_0-2,r_1+2) \\
&&&&\oplus&\\
D_{0,\sigma_4}(\rho)&:=&\sigma_4&\textbf{---}&S_4&\textbf{---}&(r_0+1,p-4-r_1)\end{array}\]
où\
\[\begin{array}{rll}
S_1:=&(p-2-r_0,r_1-1)\oplus(r_0+1,p-2-r_1) \vspace{2mm}\\
S_2:=&(r_0-2,r_1)\oplus(p-1-r_0,p-1-r_1) \vspace{2mm}\\
S_3:=&(r_0-1,p-4-r_1)\oplus(p-r_0,r_1+1) \vspace{2mm}\\
S_4:=&(p-3-r_0,p-3-r_1)\oplus(r_0,r_1+2). \vspace{2mm}
\end{array}\]
On vérifie que $\delta(\sigma_i)=\sigma_{i+1}$ (avec la convention
$\sigma_{5}:=\sigma_1$), i.e. $\sigma_i^{[s]}$ apparaît dans
$D_{0,\sigma_{i+1}}(\rho)$ comme  sous-quotient.

Notons $\chi_i:=\chi_{\sigma_i}$ et
$\chi_i^s:=\chi_{\sigma_i}^s$, et   choisissons une base
$\{e_i,e_i^{[s]}, 1\leq i\leq 4\}$ de
$D_1(\rho):=D_0(\rho)^{I_1}$, où\ $e_i$ (resp. $e_i^{[s]}$) est un
vecteur propre de $\cH$ de caractère
$\chi_i$ (resp. $\chi_i^s$). 

Fixons $D(\rho,r)$ un diagramme associé à $\rho$. Dans cette section, on va démontrer que l'ensemble $S(\rho,r)$ contient un \'el\'ement dont l'espace des $K_1$-invariants est \emph{strictement} plus grand que $D_0(\rho)$.
On fixe $\pi\in S(\rho,r)$ une repr\'esentation (supersinguli\`ere d'apr\`es \cite[th\'eor\`eme 19.10]{BP}) de $G$. 

\subsubsection{La représentation $V_1$}\label{subsubsection-V1}

Puisque $D_{0,\sigma_2}(\rho)$ contient une
sous-$K$-représentation isomorphe à l'extension
\[0\ra \sigma_2\ra *\ra \sigma_1^{[s]}\ra0\]
de type $(+1,1)$, le lemme \ref{lemma-E(deux-char)} montre
qu'il contient une sous-$I$-représentation $M_1$ isomorphe à
$E_0(\chi_2,\chi_1^s,r_0)$. Posons
\[W_1:=\Ind_I^K\Pi(M_1)\cong\Ind_I^K\Pi(E_0(\chi_2,\chi_1^s,r_0)).\] Puisque le caractère
$\chi_1^s\alpha^{-p}$ n'est autre que $\chi_3$ et puisque
$\Ind_I^K\Pi(\chi_3)$ admet un sous-quotient isomorphe à
$\sigma_4$, la sous-représentation $W_{1,\sigma_4}\subset W_1$ est
bien définie (cf. lemme \ref{lemma-W-vecteurs-k}).
On note
$V_1$ l'image de $W_{1,\sigma_4}$ dans $\pi$. Le lemme suivant
décrit la filtration par le $K$-socle de $V_1$.

\begin{lemma}\label{lemma-V_1}
(i) On a $\rsoc_KV_1=\sigma_1\oplus\sigma_3$.

(ii) La filtration par le $K$-socle de $V_1$ est la suivante (où\ l'on
pose $\tau_0:=\sigma_3$ et $\tau_1:=\sigma_2^{[s]}$):
\[\begin{array}{rccccccccccccccccccc}\tau_0&\ligne&\tau_1&\ligne&\cdots&\ligne&\tau_{r_0-1}&\ligne&
\tau_{r_0}&\ligne&{\tau}_{r_0-1}&\ligne&\cdots&\ligne&\tau_0&\ligne&\sigma_4\\
&&&&&&&&&&&&&&&&\mid\\
&&&&&&&&&&&&&&&&\sigma_1.\end{array}\] De manière explicite,
$(V_1/\sigma_1)_i=
\left\{ {\begin{array}{ll}\tau_i& \mathrm{si\ }i\leq r_0\\

\tau_{2r_0-i}& \mathrm{si\ }r_0+1\leq i\leq 2r_0 \\

\sigma_4 & \mathrm{si\ }i=2r_0+1.
\end{array}}\right.$

(iii)  $\tau_i\notin\cD(\rho)$ pout tout $i\neq 0$; le couple
$(\sigma_1,\sigma_4)$ est de type $(+1,1)$ et
$(\tau_i,\tau_{i+1})$ de type $(+1,0)$ pour tout $0\leq i\leq r_0-1$.
\end{lemma}
\begin{proof}
C'est une conséquence du corollaire \ref{corollary-W-U(tau)} (analogue \`a l'exemple \ref{exemple-1}). D'abord, on constate que l'image de $\Ind_I^K\Pi(\chi_2)$ dans $\pi$ est $I(\sigma_3,\sigma_2^{[s]})$ et celle de $\Ind_I^K\Pi(\chi_1^s)$ est $\sigma_1$. Puis, en prenant $j=0$ dans l'exemple \ref{exemple-1}, on v\'erifie qu'aucun des poids $\sigma^{(i)}_{\emptyset}$ et $\sigma^{(i)}_{\{0\}}$ pour $1\leq i\leq t$ n'appartient \`a $\cD(\rho)$. Le r\'esultat s'en d\'eduit en rappelant que seul les poids de Diamond peuvent appara\^itre comme sous-$K$-repr\'esentation de $\pi$.
\end{proof}

Par (\ref{equation-def-vecteur-W}) de la preuve du lemme \ref{lemma-W-vecteurs-k}, $V_1$ est engendr\'ee par
le vecteur
\[F_{p(p-2-r_1)}=\summ_{\lambda\in\F_q}\lambda^{p(p-2-r_1)}\matr{[\lambda]}110\Pi(v_1)\]
o\`u $v_1$ est un vecteur non nul propre de $\cH$ de caract\`ere $\chi_3$. On note $S_1\subset V_1$ la sous-$K$-repr\'esentation engendr\'ee par le vecteur
\[F_{0}=\summ{\lambda\in\F_q}\matr{[\lambda]}110\Pi(v_1).\]
Alors, encore par (\ref{equation-def-vecteur-W}), $S_1$ n'est autre que $W_{\sigma_3}$ en regardant $\sigma_3$ comme un sous-quotient de $\Ind_I^K\Pi(\chi_3)$. Par le corollaire \ref{corollary-W-U(tau)} et le lemme \ref{lemma-V_1}, on trouve la filtration par le socle de $S_1$:
\begin{equation}\label{equation-S1} \tau_0\ \ligne\ \tau_1 \ \ligne\ \cdots\ \ligne\
\tau_{r_0}\ \ligne\ \cdots\ \ligne \ \tau_1\ \ligne\
\tau_0.\end{equation}

\subsubsection{La représentation $V_2$}\label{subsubsection-V2}
On va construire un autre sous-espace vectoriel $V_2$ de ${\pi}$
qui est $K$-stable et qui possède la même filtration par le socle que $V_1$.

\begin{lemma}\label{lemma-exemple-f=2-sous-rep}
Il existe une unique sous-$K$-représentation de $\pi$ qui est isomorphe à
\[\begin{array}{ccccc}\sigma_4&\ligne&\sigma_3^{[s]}&\ligne&\omega:=(p-2-r_0,r_1+3)
\otimes{\det}^{r_0+p(p-2)}.\end{array}\]
\end{lemma}
\begin{proof}
On vérifie que $\sigma_3$ contient une sous-$I$-représentation
isomorphe à $E_0(\chi_3)$ (grâce à l'hypothèse $r_0\leq p-2$) et
que $\Ind_I^K\Pi(\chi_3\alpha^{-1})$ admet le poids $\omega$ comme
sous-quotient. Soit $W_{\omega}$ la sous-$K$-représentation de
$\Ind_I^K\Pi(E_0(\chi_3))$ définie dans la proposition
\ref{prop-W-vecteurs}. Le lemme \ref{lemma-Breuil} implique que $W_{\omega}$ contient enti\`erement $\Ind_I^K\Pi(\chi_3)$. D'autre part, par construction de $D_0(\rho)$, on sait que l'image de $\Ind_I^K\Pi(\chi_3)$ dans $\pi$ est isomorphe \`a $I(\sigma_4,\sigma_3^{[s]})\subset D_{0,\sigma_4}(\rho)$. Comme le socle de $\Ind_I^K\Pi(\chi_3\alpha^{-1})$ qui appara\^it aussi dans $W_{\omega}$ n'est pas un poids de Diamond, on trouve que  l'image de
$W_{\omega}$ dans ${\pi}$ v\'erifie la condition demand\'ee. L'unicit\'e se d\'eduit de la proposition \ref{prop-extension-K} (ii) et du fait que $\sigma_3^{[s]},\omega\notin\cD(\rho)$.
\end{proof}

On note $\overline{W_{\omega}}$ la $K$-représentation construite dans le lemme
\ref{lemma-exemple-f=2-sous-rep}. Choisissons une base $\{v,w\}$
pour $E_0(\chi_3)$ et définissons des vecteurs $f_k$, $F_k$
($0\leq k\leq q-1$) dans $\Ind_I^K\Pi(E_0(\chi_3))$ comme dans
(\ref{equation-define-F-et-f}). Alors, vu comme sous-quotient
de $\Ind_I^K\Pi(E_0(\chi_3))$, $\overline{W_{\omega}}$ admet une
base induite (\cite[lemme 2.7]{BP})
\[\left\{{\begin{array}{lll}\overline{f}_{d_0+pd_1}, &p-2-r_1\leq d_1\leq p-1, \mathrm{\ ou\ }
d_1= p-3-r_1 \mathrm{\ et\ }  p-1-r_0\leq d_0\leq p-1\\
\overline{F}_{d_0'+pd_1'},& 0\leq d_0'\leq p-2-r_0 \mathrm{\ et\
}p-4-r_1\leq d_1'\leq p-1\end{array}}\right\}\] où $\overline{f}_k$ (resp.
$\overline{F}_k$) désigne l'image de $f_k$ (resp. $F_k$) dans
$\overline{W_{\omega}}\subset\pi$. On constate que $\overline{f}_{p(p-2-r_1)}$ et $\overline{f}_{(p-1-r_0)+p(p-3-r_1)}$ engendrent l'espace des $I_1$-invariants de $\overline{W_{\omega}}$: en fait, on a $\overline{f}_{p(p-2-r_1)}\in\bFp e_4$ et $\overline{f}_{(p-1-r_0)+p(p-3-r_1)}\in\bFp e_3^{[s]}$. Notons que le vecteur $\overline{F}_{p(p-4-r_1)}$ est fix\'e par $\smatr{1+\p}{\cO_F}{\p^2}{1+\p}$ mais pas fix\'e par $I_1$: on a
\begin{equation}\label{equa-exemple-e_omega}
\matr10p1\overline{F}_{p(p-4-r_1)}=\overline{F}_{p(p-4-r_1)}-\overline{f}_{p(p-2-r_1)}
\end{equation} par le lemme \ref{lemma-Witt-dans-W} (iii).\vv

Le vecteur $v_{\omega}:=\overline{F}_{p(p-3-r_1)}\in \overline{W_{\omega}}$ est un vecteur
propre de $\cH$ de caractère $\chi_{\omega}\alpha^{-p}=\chi_4$ (par le lemme \ref{lemma-Witt-dans-W} (i)). Posons $M_2\subset \pi$ la $I$-représentation
engendrée par $v_{\omega}$. Explicitement, $M_2$ admet une base form\'ee par:
\begin{equation}\label{equation-base-M_2}
\left\{{\begin{array}{lll}\overline{f}_{d_0+pd_1},\ \ \ d_1=p-3-r_1\ \mathrm{et}\ p-1-r_0\leq d_0\leq p-1&
\\
\overline{f}_{p(p-2-r_1)},\ \overline{f}_{p(p-1-r_1)},\ \overline{F}_{p(p-4-r_1)},\ \overline{F}_{p(p-3-r_1)}&\end{array}}
\right\}.\end{equation}
Pour définir l'espace $V_2$, nous avons besoin de connaître
la structure de $M_2$.

D\'efinissons d'abord un sous-espace vectoriel $M_2'$ de $M_2$ par $M_2':=\bFp e_4\oplus \bFp
e_{\omega}\subset M_2$, o\`u l'on \'ecrit $e_{\omega}:=\overline{F}_{p(p-4-r_1)}$ pour simplifier (cette notation vient du fait que son image dans $\omega$ engendre $\omega^{I_1}$). Vu l'\'equation (\ref{equa-exemple-e_omega}), $M_2'$ est stable par $I$ et isomorphe \`a $\Pi(E_1(\chi_4^s))$, i.e. \`a
l'extension
\[0\ra \chi_4\ra *\ra \chi_4\alpha^p\ra 0.\]

\begin{lemma}\label{lemma-exemple-M_2''}
La sous-$K$-repr\'esentation $\langle K\cdot \Pi(M_2')\rangle\subset\pi$ est isomorphe \`a $I(\sigma_1,\sigma_4^{[s]})$.
%
\end{lemma}
\begin{proof}
En voyant $\langle K\cdot\Pi(M_2')\rangle$ comme l'image du morphisme $\Ind_I^K\Pi(M_2')\ra \pi$, il s'agit de d\'eterminer le noyau $\cK$ de ce morphisme. Par la preuve du
\cite[lemme 19.5]{BP}, la sous-$K$-repr\'esentation de $\pi$ engendr\'ee par $\Pi(e_4)$ est isomorphe \`a $I(\sigma_1,\sigma_4^{[s]})$. Donc, si l'on note $\cK_1$ le noyau du
quotient
\[\Ind_I^K\bFp\Pi(e_4)\twoheadrightarrow
I(\sigma_1,\sigma_4^{[s]}),\] alors $\cK_1\subset\cK$ et on obtient une injection $I(\sigma_1,\sigma_4^{[s]})\hookrightarrow\Ind_I^K\Pi(M_2')/\cK_1$. D'autre part, on peut v\'erifier que $I(\sigma_1,\sigma_4^{[s]})$ contient une sous-$I$-repr\'esentation isomorphe \`a $E_1(\chi_4^s)$, donc il existe par r\'eciprocit\'e de Frobenius une surjection $K$-\'equivariante $\Ind_I^K\Pi(M_2')\cong\Ind_I^KE_1(\chi_4^s)\twoheadrightarrow I(\sigma_1,\sigma_4^{[s]})$ qui envoie forc\'ement $\cK_1$ vers $0$ par ce qui pr\'ec\`ede. On en d\'eduit une d\'ecomposition:
\[\Ind_I^K\Pi(M_2')/\cK_1\cong I(\sigma_1,\sigma_4^{[s]})\oplus\Ind_I^K\bFp\Pi(e_{\omega}).\]
Or, en utilisant le lemme \ref{lemma-Gamma-PS}, on vérifie facilement que $\Ind_I^K\bFp\Pi(e_{\omega})$
n'admet aucun poids de Diamond comme sous-quotient. Cela force que $\Ind_I^K\bFp\Pi(e_{\omega})\subset\cK$ puisque $\rsoc_K\pi=\oplus_{\sigma\in\cD(\rho)}\sigma$ et le r\'esultat s'en d\'eduit.
\end{proof}

Ensuite, posons $M_2''$ la sous-$I$-repr\'esentation de $M_2$ engendr\'ee par $\overline{f}_{p-1+p(p-3-r_1)}$. En utilisant la proposition \ref{prop-U+-action} (i), $M_2''$ est le sous-espace vectoriel de $M_2$ de dimension $r_0+1$ engendr\'e par (cf. (\ref{equation-base-M_2})):
\[\Bigl\{\overline{f}_{d_0+pd_1},\ \ \ d_1=p-3-r_1\ \mathrm{et}\ p-1-r_0\leq d_0\leq p-1\Bigr\}.\]
Autrement dit, $M_2''$ est isormorphe \`a $E_0(\chi_3^s,r_0)$.

\begin{lemma}\label{lemma-exemple-M_2'}
La $K$-représentation $\langle K\cdot \Pi(M_2'')\rangle\subset\pi$
contient la repr\'esentation $S_1$.
\end{lemma}
\begin{proof}
Posons $W_2'':=\Ind_I^K\Pi(M_2'')$. Puisque
\[\rcosoc_I(M_2')=\chi_3^s\alpha^{-r_0}\cong \chi_4\alpha=(p-r_0,r_1+1)\otimes{\det}^{(r_0-1)+p(p-1)}\]
et que $\Ind_I^K\Pi(\chi_4\alpha)$ admet $\sigma_3=\tau_0$ comme
sous-quotient (avec $J(\sigma_3)=\{1\}$), la sous-repr\'esentation
$W_{2,\sigma_3}''\subset W_{2}''$ est d\'efinie. Le corollaire \ref{corollary-W-U(tau)} (combin\'e avec l'exemple \ref{exemple-1}) montre que $W_{2,\sigma_3}''$ a la m\^eme filtration par le socle que $S_1$. De plus, par un calcul compliqu\'e mais direct, on trouve que $W_{2,\sigma_3}''$ est isomorphe \`a $S_1$ en tant que $K$-repr\'esentation. En effet, l'isomorphisme est induit par envoyant $F_{0}$ vers le vecteur
\[\summ_{\lambda\in\F_q}\lambda^{p(r_1+2)}\matr{[\lambda]}110\Pi(w_{\chi_4\alpha})\]
avec $w_{\chi_4\alpha}\in M_2''$ un vecteur non nul propre de $\cH$ de caract\`ere $\chi_4\alpha$.

Apr\`es, on montre que $W_{2,\sigma_3}''$ co\"{i}ncide avec $S_1$ en tant que \emph{sous-espaces} vectoriels de $\pi$. En effet, cela r\'esulte des faits que $S_1$ n'admet pas de poids de Diamond autre que $\sigma_3$ comme sous-quotient et que $\sigma_3$ appara\^it dans $\rsoc_K(\pi)$ avec multiplicit\'e 1.
\end{proof}\vv

D'autre part, notons $\overline{M_2}$ le quotient de $M_2$ par $M_2'$ de sorte que $\langle K\cdot\overline{M_2}\rangle$ s'injecte dans $\pi/I(\sigma_1,\sigma_4^{[s]})$.
Explicitement, $\overline{M_2}$ est une extension de $\Pi(E_j(\chi_4^s\alpha^{p}))$ par $M_2''$ dont la filtration par le socle est:
\begin{equation}\label{equation-M_2-bar}\begin{array}{rccccccccccccccc}\chi_3^s&\ligne&\chi_3^s\alpha^{-1}&\ligne&\cdots&\cdots&\ligne&\chi_3^s\alpha^{-r_0}&\ligne
&\chi_3^s\alpha^{-(r_0+1)}=\chi_4\\
&&&&&&&&&\vdots\\
&&&&&&&&&\chi_4\alpha^{-p}\end{array}\end{equation}
o\`u l'on utilise la notion $\chi_4\alpha^{-p}\cdots \chi_4$ pour pr\'eciser que cette $I$-extension est triviale sur $\smatr{1+\p}{\cO_F}{\p^2}{1+\p}$.

\begin{prop}\label{lemma-exemple-V_1=V_2}
La $K$-représentation $\langle K\cdot \Pi(M_2)\rangle\subset \pi$
 contient une sous-$K$-représentation $V_2$ qui contient $S_1$ et qui a même
filtration par le socle que $V_1$.
\end{prop}
\begin{proof}
Rappelons que $v_{\omega}$ est le vecteur engendrant $M_2$. Posons \[R_0=\sum_{\lambda\in\F_q}\matr {[\lambda]}110
\Pi(v_{\omega})\in \langle K\cdot\Pi(M_2)\rangle.\]
Vu la structure de $\overline{M_2}$ (\ref{equation-M_2-bar}), le lemme \ref{lemma-Ind(Ej)} (i) montre que l'image de $R_0$ dans $\Ind_I^KE_j(\chi_4^s\alpha^{p})$ est fix\'ee par $I_1$ et engendre une $K$-repr\'esentation isomorphe \`a $\sigma_4$. Puis, on applique le corollaire \ref{corollary-W-U(tau)} pour d\'eduire que l'image de $S_2:=\langle K\cdot R_0\rangle$ dans $\Ind_I^K\overline{M_2}$ contient une extension de $\sigma_4$ par $S_1$ (cf. le lemme \ref{lemma-exemple-M_2'} et (\ref{equation-M_2-bar})), i.e.
\[ \tau_0\ \ligne\ \tau_1 \ \ligne\ \cdots\ \ligne\
\tau_{r_0}\ \ligne\ \cdots\ \ligne \ \tau_1\ \ligne\
\tau_0\ \ligne\ \sigma_4.\]
Combin\'e avec le lemme \ref{lemma-exemple-M_2''}, on voit que $\langle K\cdot\Pi(M_2)\rangle$  contient une $K$-extension $S$ de la forme:
\[0\ra I(\sigma_1,\sigma_4^{[s]})\ra S\ra S_2\ra0\]
qui provient effectivement d'un certain \'el\'ement dans $\Ext^1_K(\sigma_4,I(\sigma_1,\sigma_4^{[s]}))$.
Par ailleurs, en utilisant le lemme \ref{lemma-V_1}, on v\'erifie que $\sigma_4$ appara\^it dans $S$ avec multiplicit\'e 1. Par cons\'equent, il existe une unique sous-representation $U(\sigma_4)\subset S$ admettant $\sigma_4$ comme cosocle.
Comme $\Ext^1_{K}(\sigma_4,\sigma_4^{[s]})=0$, on voit que $\sigma_4^{[s]}$ n'appara\^it pas dans $U(\sigma_4)$ et pour conclure il suffit de poser $V_2=U(\sigma_4)+\sigma_1$.
\end{proof}

\subsubsection{Conclusion}
Rappelons que l'on a défini deux sous-espaces $M_1$ et $M_2$ de
$\pi$. Choisissons une repr\'esentation lisse admissible $\Omega$ de $G$ telle que $\pi\hookrightarrow \Omega$ et telle que $\Omega|_K\cong\oplus_{\sigma\in\cD(\rho)}\rInj_K\sigma$ et posons $M=M_1+M_2$. Rappelons que $v_1\in M_1$ est un vecteur propre de $\cH$ de caract\`ere $\chi_3$ engendrant $M_1$. On a les faits suivants sur $(M,v_{1})$: 
\begin{enumerate}
\item[(S1)] $v_{1}\notin \Sigma(M)$. En effet, la longueur de Loewy de $M_1$ (resp. $M_2$) en tant que $U^+$-repr\'esentation est \'egale \`a $r_0+1$ par la remarque \ref{remark-E=multi-1} (resp. $r_0+2$) et la longueur de Loewy de $M_2$ en tant que $U^-$-repr\'esentation est \'egale \`a 2. Donc, si l'on a par absurde $v_1\in \Sigma(M)$, alors le caract\`ere $\chi_3$ devrait appara\^itre dans $M_2/\rsoc_I^{r_0}(M_2)$ qui est isomorphe \`a $\chi_4\oplus\chi_3^{s}\alpha^{-r_0}$, d'o\`u une contradiction.

\item[(S2)] Il existe $f_{\sigma_3}\in(\rInj_K\sigma_3)^{I_1}$ un vecteur non nul propre de $\cH$ de caract\`ere $\chi_3^s$: c'est clair par le lemme \ref{lemma=S1S2}.
\end{enumerate}
On obtient alors par la construction de \S\ref{subsection-prelimilaire} une famille de repr\'esentations lisse admissibles de $G$, $\{\pi_{\phi,a},\ {\phi\in\Phi, a\in\bFp}\}$.
\begin{lemma}
Pour tout $\phi\in\Phi$ et $a\in\bFp$, on a $\pi_{\phi,a}\in S(\rho,r)$.
\end{lemma}
\begin{proof}
En examinant la preuve du lemme \ref{lemma-exemple-f=3-in-S}, on trouve qu'il suffit de d\'emontrer que $M_2$ s'injecte dans $\pi_{\phi,a}$, ou encore la $K$-repr\'esentation $\overline{W_{\omega}}$ d\'efinie dans le lemme \ref{lemma-exemple-f=2-sous-rep} est contenue dans $\pi_{\phi,a}$. Or, en reprenant les notations de ce lemme, cela r\'esulte du fait que tout quotient de $\Ind_I^KE_0(\chi_3)$ dans $\pi$ contient $\overline{W_{\omega}}$ comme sous-repr\'esentation (comparer avec \cite[lemme 19.5]{BP}).
\end{proof}

Le théorème suivant va répondre \emph{négativement} à la question (Q2).
\begin{theorem}\label{theorem-exem-f=2}
Il existe un \'el\'ement $\pi'$ de $S(\rho,r)$ telle que $\pi'^{K_1}\supsetneq D_0(\rho)$. 
\end{theorem}
\begin{proof}
Supposons que la repr\'esentation $\pi$ que l'on a fix\'e soit telle que $\pi^{K_1}=D_0(\rho)$. On va
construire un autre \'el\'ement de $S(\rho,a)$
vérifiant la condition demandée.

Comme $V_1$ (et de m\^eme pour $V_2$) n'admet pas de poids de Diamond autre que $\sigma_3$ et $\sigma_4$ et comme $\Ext^1_K(\sigma_4,\tau)=0$ pour tout sous-quotient irr\'eductible de $S_1$ autre que $\sigma_3$, la condition $\pi^{K_1}=D_0(\rho)$ force que $V_1=V_2$ en tant que sous-espaces de $\pi$ car sinon $\pi$ contiendrait une extension non triviale de la forme
\[0\ra \sigma_3\oplus\sigma_1\ra *\ra \sigma_4\ra0,\]
ce qui est impossible parce que une telle extension est triviale sur $K_1$ et $D_0(\rho)$ n'en contient pas.

Par d\'efinition de $\Pi_{\phi,a}$, on a  $\Pi_{\phi,a}(v_1)=\Pi(v_1)+af_{\sigma_3}$ et
$\Pi_{\phi,a}(x)=\Pi(x)$ pour $x\in M_2$, donc  $\pi_{\phi,a}$ contient \`a la fois les vecteurs $F_{p(p-2-r_1)}$ (provenant de $V_2$!) et
$F_{p(p-2-r_1)}+af_{p(p-2-r_1)}$ o\`u \[f_{p(p-2-r_1)}=\summ_{\lambda\in\F_q}\lambda^{p(p-2-r_1)}\matr{[\lambda]}110f_{\sigma_3}.\]
Cela montre que, d\`es que $a\neq0$, $\pi_{\phi,a}$ contient l'extension $0\ra \sigma_3\ra *\ra\sigma_4\ra0$ (engendr\'ee par $f_{p(p-2-r_1)}$) et il suffit de poser $\pi'=\pi_{\phi,a}$ avec $a\neq 0$ pour conclure.
\end{proof}

  \hspace{5cm} \hrulefill\hspace{5.5cm}

Université Paris-Sud Mathématiques, Bâtiment 425, 91400, Orsay, France

E-mail: yongquan.hu@math.u-psud.fr

\end{document}